\documentclass[11pt]{article}

\usepackage{amsmath, amsthm, amscd, amssymb, xcolor, fancyhdr, mathrsfs}
\usepackage[scr=boondox]{mathalfa}
\definecolor{dblue}{rgb}{0,0,.6}
\usepackage{silence}
\usepackage[backref=page, colorlinks=true, linkcolor=dblue, citecolor=dblue, filecolor=dblue, menucolor=dblue, urlcolor=dblue]{hyperref}
\usepackage{tikz-cd}

\renewcommand*{\backrefalt}[4]{
	\ifcase #1 (Not cited)
	\or        (Cited on page~#2)
	\else      (Cited on pages~#2)
	\fi}

\newcommand{\6}{\partial}
\newcommand{\Q}{\mathbb{Q}}
\newcommand{\R}{\mathbb{R}}
\newcommand{\Z}{\mathbb{Z}}
\newcommand{\C}{\mathbb{C}}

\def\cal{\mathcal}

\newcommand{\arrow}{{\:\longrightarrow\:}}

\newcommand{\bbZ}{\mathbb{Z}}
\newcommand{\bbQ}{\mathbb{Q}}
\newcommand{\bbR}{\mathbb{R}}
\newcommand{\bbC}{\mathbb{C}}
\newcommand{\bbH}{\mathbb{H}}

\newcommand{\bbP}{\mathbb{P}}

\newcommand{\CC}{\mathcal{C}}
\newcommand{\DD}{\mathcal{D}}
\newcommand{\EE}{\mathcal{E}}
\newcommand{\FF}{\mathcal{F}}
\newcommand{\HH}{\mathcal{H}}
\newcommand{\II}{\mathcal{I}}
\newcommand{\KK}{\mathcal{K}}

\newcommand{\PP}{\mathcal{P}}

\newcommand{\VV}{\mathcal{V}}

\newcommand{\XX}{\mathcal{X}}

\newcommand{\frh}{\mathfrak{h}}

\newcommand{\frr}{\mathfrak{r}}

\newcommand{\frso}{\mathfrak{so}}

\newcommand{\SO}{\operatorname{SO}}
\newcommand{\Sp}{\operatorname{Sp}}

\newcommand{\rmO}{\operatorname{O}}

\renewcommand{\ge}{\geqslant}
\renewcommand{\le}{\leqslant}

\newcommand{\NS}{\mathrm{NS}}

\newcommand{\st}{\enskip |\enskip}
\newcommand{\sdot}{{\raisebox{0.16ex}{$\scriptscriptstyle\bullet$}}}

\newcommand{\ii}{\sqrt{-1}}

\newcommand{\wdg}{\wedge}

\newcommand{\Aut}{\mathrm{Aut}}
\newcommand{\Amp}{{\mathcal{A}\it {mp}}}
\newcommand{\Av}{\mathrm{Av}}

\newcommand{\Vol}{\mathrm{Vol}}
\newcommand{\VVol}{\mathcal{V}\mathrm{ol}}

\newcommand{\Comp}{\mathscr{I}}

\newcommand{\HK}{\mathscr{H\!K}}
\newcommand{\HKa}{\HK^+}
\newcommand{\HKb}{\mathscr{H\!K\!S}}

\newcommand{\Teich}{\mathcal{T}}
\newcommand{\TeichP}{\Teich_p}
\newcommand{\TeichHK}{\Teich_{h\!k}}
\newcommand{\TeichHKa}{\Teich_{h\!k^+}}
\newcommand{\TeichHKb}{\Teich_{h\!k\!s}}

\newcommand{\MC}{\mathcal{MCG}}
\newcommand{\Diff}{\mathscr{D\hspace{-0.15em}i\hspace{-0.3em}f\hspace{-0.4em}f}}
\newcommand{\Dom}{\mathcal{D}}
\newcommand{\DomP}{\Dom_p}
\newcommand{\DomHK}{\Dom_{h\!k}}
\newcommand{\DomHKa}{\Dom_{h\!k^+}}
\newcommand{\DomHKb}{\Dom_{h\!k\!s}}

\newcommand{\Per}{\rho}
\newcommand{\PerP}{\Per_p}
\newcommand{\PerHK}{\Per_{h\!k}}
\newcommand{\PerHKa}{\Per_{h\!k^+}}
\newcommand{\PerHKb}{\Per_{h\!k\!s}}

\newcommand{\MBM}{\mathrm{MBM}}

\newtheorem{defn}{Definition}[section]
\newtheorem{prop}[defn]{Proposition}
\newtheorem{thm}[defn]{Theorem}
\newtheorem{lem}[defn]{Lemma}
\newtheorem{cor}[defn]{Corollary}

{\theoremstyle{remark}
  \newtheorem{rem}[defn]{Remark}
  
}

\makeatletter
\@addtoreset{equation}{section}
\makeatother

\newcommand{\version}{Version 3.1 (arxiv), Oct. 21, 2024}
\begin{document}

%\title{Rigid currents on compact hyperk\"ahler manifolds}
%\date{}

\begin{center}
{{\LARGE Rigid currents on compact hyperk\"ahler manifolds}\\[6mm]}
Nessim Sibony$^\dagger$,
Andrey Soldatenkov,
%\footnote{The work of Andrey Soldatenkov was performed
%at the Steklov International Mathematical Center and supported
%by the Ministry of Science and Higher Education of the Russian Federation
%(agreement no. 075-15-2019-1614)},
Misha Verbitsky\footnote{Misha Verbitsky is
partially supported by FAPERJ 	
SEI-260003/000410/2023 and CNPq - Process 310952/2021-2.

}
\end{center}

{\small\noindent\hfill
\begin{minipage}[t]{0.8\linewidth}
{\bf\hfill Abstract \hfill\phantom{}}\\
A rigid cohomology class on a complex manifold
is a class that is represented by a unique
closed positive current. The positive current representing a rigid
class is also called rigid.
For a compact K\"ahler manifold $X$ all eigenvectors of hyperbolic
automorphisms acting on $H^{1,1}(X)$ that have non-unit eigenvalues
are rigid classes. Such classes are always parabolic,
namely, they belong to the boundary of the K\"ahler cone
and have vanishing volume. We study parabolic $(1,1)$-classes
on compact hyperk\"ahler manifolds with $b_2 \ge 7$.
We show that a parabolic class is rigid if it is not orthogonal to a rational vector with respect to the BBF form.
This implies that a general parabolic class on a hyperk\"ahler manifold is rigid.
\end{minipage}\hfill}

\setlength{\headheight}{15pt}
\fancypagestyle{plain}{
  \fancyhf{}
  \cfoot{\small -- \thepage \ --} \rfoot{\tiny \sc\version}
  \renewcommand{\headrulewidth}{0pt}%
}
\thispagestyle{plain}

\setcounter{tocdepth}{2}
{\small
\tableofcontents
}

%%%%%%%%%%%%%%%%%%%%%%%%%%%%%%%%%%%%%%%%%%%%%%%%%%%%%%%%%%%%%%%%%%%%%%%%

\section{An introduction to rigid currents}

%%%%%%%%%%%%%%%%%%%%%%%%%%%%%%%%%%%%%%%%%%%%%%%%%%%%%%%%%%%%%%%%%%%%%%%%

%%%%%%%%%%%%%%%%%%%%%%%%%%%%%%%%%%%%%%%%%%%%%%%%%%%%%%%%%%%%%%%%%%%%%%%%
\subsection{Rigid currents in complex dynamics}
%%%%%%%%%%%%%%%%%%%%%%%%%%%%%%%%%%%%%%%%%%%%%%%%%%%%%%%%%%%%%%%%%%%%%%%%

{\em Rigid currents} on complex manifolds
are positive closed $(p,p)$-currents that admit a unique
positive closed representative in their cohomology class.
More specifically, let $X$ be a compact K\"ahler manifold of dimension $k$.
Consider a cohomology class $\alpha \in H^{p,p}_\bbR(X)$, where the subscript $\bbR$
means that $\bar{\alpha} = \alpha$. The question that we are interested
in is whether it is possible to represent $\alpha$ by a closed positive current
and to what extent such a representation is unique (see e.g. \cite{Dem} for
the definition of positive currents and an overview of the underlying theory).
Thinking of currents as of differential forms whose coefficients are distributions,
we denote the space of currents of Hodge type $(p,q)$ by $\EE'(X)^{p,q}$. For
a current $T\in \EE'(X)^{p,p}_\bbR$ we will write $T\ge 0$ and call $T$ positive if 
for any $\eta_1,\ldots,\eta_{k-p}\in \Lambda^{1,0}X$ we have
$$
(\ii)^{k-p}\langle T, \eta_1\wdg\bar{\eta}_1\wdg\ldots\wdg \eta_{k-p}\wdg\bar{\eta}_{k-p}\rangle \ge 0.
$$
For a closed current $T$ we denote by $[T]$ its cohomology class. Define
$$
\CC_\alpha = \{ T\in \EE'(X)_\bbR^{p,p}\st T\ge 0,\,dT=0,\,[T]=\alpha\}.
$$

The class $\alpha$ is called {\em rigid} if $\CC_\alpha$ contains exactly one element.
The unique positive closed current representing a rigid class is also called rigid.
Restricting to the case $p=1$, recall that a class $\alpha\in H^{1,1}_\bbR(X)$ is called {\em pseudo-effective} if $\CC_\alpha \neq \emptyset$.
A pseudo-effective class is called {\em nef} if it lies in the closure of the K\"ahler cone of $X$.

%The pseudo-effective classes form a closed convex cone $\PP_X \subset H^{1,1}_\bbR(X)$.
%Note that a rigid class $\alpha$ spans an extremal ray of $\PP_X$: if $\alpha = \alpha_1 + \alpha_2$
%where $\alpha_i\in \PP_X$ and $T_i$ is a closed positive current with $[T_i] = \alpha_i$,
%then $[T_1 + T_2] = [2T_1 ]$

To put our work in context, let us recall how rigid currents appear in holomorphic dynamics, see e.g. \cite{DS3}.
For a holomorphic map $f\colon X\to X$ the $p$-th dynamical degree $d_p(f)$
is defined as the spectral radius of $f^*$ acting on $H^{p,p}(X)$.
It follows from \cite{Gr} that the function $p\mapsto \log(d_p(f))$ is concave
and the sequence of dynamical degrees is of the following form:
$$
1=d_0 < d_1 <\ldots <d_{p_0} =\ldots = d_{p_1} > \ldots > d_{k-1} > d_k = 1.
$$
We will call a non-zero closed positive current $T\in \EE'(X)^{p,p}_\bbR$ a {\em Green current} for $f$ of order $p$
if $f^*T = d_p(f) T$. Recall the following result which is a special case
of \cite[Theorem 4.2.1]{DS3} combined with \cite[Theorem 4.3.1]{DS3}.
\begin{thm}[Dinh--Sibony, \cite{DS3}]
Let $f$ be a holomorphic automorphism of a compact K\"ahler manifold $X$
and $d_j(f)$ the dynamical degrees of $f$.
Assume that $d_{p-1}(f) < d_p(f)$ for some $p\ge 1$. Then there exists a Green
current $T$ of order $p$ for $f$ and this current is rigid.
\end{thm}
This direction of research originates from the work of Cantat \cite{_Cantat:K3_}, \cite{_Cantat:K3-2_}, who
has discovered and studied dynamical currents on K3 surfaces.
We refer to \cite{DS3} for much more precise and general results of this form,
and to the references in loc. cit. for an overview of how rigid currents arise
in other settings.

Note that when the above theorem applies to $p=1$ and $T$ is the corresponding
Green $(1,1)$-current, its cohomology
class $\alpha = [T]\in H^{1,1}_\bbR(X)$ satisfies $\alpha^k = 0$, where $k = \dim_\bbC (X)$.
This follows from the condition $d_1(f) > 1$ and the fact that the automorphism $f$ acts as the identity
on the top degree cohomology of $X$. We introduce the following terminology:
a cohomology class $\beta\in H^{1,1}_\bbR(X)$ is called {\it parabolic} if
it is nef and $\beta^k = 0$.

\subsection{Rigid currents on complex tori}

We give an example of a rigid current which is
smooth, unlike almost all rigid currents considered later in
this paper. This subsection is included here to illustrate the
concept of rigidity with the most elementary and hands-on example.
However, the mapping class group argument (Subsection
\ref{_orbits_parabolic_Subsection_}) which is
central for our work,  can be applied to a compact
torus as well, bringing another proof of the following
proposition. 

%\hfill

%%%%%%%%%%%%%%%%%%%%%%%%%%%%%%%%%%%%%%%%%%%%%%%%
\begin{prop}\label{_rigid_on_torus_Proposition_}
Consider a translation-invariant
holomorphic foliation ${\cal F}\subset TM$ on
a compact complex torus $M$ with $\mathrm{rk}(\FF) = d$. Suppose that
$\omega$ is a closed positive (1,1)-form vanishing on ${\cal F}$
and strictly positive in the transversal directions.
Assume that one (and hence all) leaves of
${\cal F}$ are dense in $M$. Then $\omega$ is rigid.
\end{prop}
\begin{proof}{\em Step 1.}
Averaging $\omega$
with respect to translations, we may assume that
$\omega$ is translation invariant. We are going to show that any other
current $\omega_1$ in the same cohomology class also vanishes
on ${\cal F}$, that is, satisfies $i_X \omega_1=0$ for
any vector field tangent to ${\cal F}$. Indeed, if
$\omega_1$ does not vanish on ${\cal F}$,
it follows that the distribution
$i_{X} i_{I(X)} \omega_1$ is non-zero for some translation
invariant vector field $X$ tangent to ${\cal F}$,
where $I$ denotes the complex structure on $M$.
We replace $\omega_1$ with the average $\Av(\omega_1)$
of its translations by the torus.
We obtain a smooth form which satisfies $i_{X} i_{I(X)} \omega_1>0$
for some translation invariant vector field $X$ tangent to
${\cal F}$. The form $\omega_1$ is in the same cohomology class as $\omega$
and satisfies $(\omega+ \omega_1)^{d+1}\neq 0$,
because $\omega+\omega_1$ is strictly positive in 
the direction transversal to $\FF$, and non-zero
somewhere in $\FF\subset TM$. Since the form $(\omega+
\omega_1)^{d+1}$
is positive, its cohomology class is non-zero
(indeed, we can multiply this form by an appropriate
power of a K\"ahler form to obtain a positive 
volume form, non-zero somewhere on $M$).
However, $[\omega_1] = [\omega]$,
hence $[\omega+\omega_1]^{d+1} = [2\omega]^{d+1} = 0$.
Therefore, all positive closed $(1,1)$-currents in the cohomology
class of $\omega$ vanish on ${\cal F}$.

{\em Step 2.}
Let $\omega_1$ be a current in the same cohomology class
as $\omega$. Then $\omega_1$ also vanishes on ${\cal F}$ by Step 1.
By the $dd^c$-lemma, $\omega_1 - \omega= dd^c f$ for some distribution $f$
which can be locally obtained as the sum of a plurisubharmonic
and a smooth function. By  \cite[Theorem A.5]{HL}, $f$ is upper semicontinuous,
hence it reaches its maximum somewhere on $M$. On the corresponding
leaf of ${\cal F}$, the function $f$ is constant by the maximum
principle; since this leaf is dense in $M$, $f$ is constant and $\omega_1 = \omega$.
\end{proof}

When the foliation ${\cal F}$ is obtained as the
unstable foliation of an Anosov automorphism of a torus,
\ref{_rigid_on_torus_Proposition_} was already shown
in \cite[Remarque 2.3]{_Cantat:K3_}.

In the setting of complex surfaces,
the rigid currents supported on 1-dimensional complex
foliations (and even laminations) were studied
extensively in \cite{_Dinh_Nguen_Sibony_}, obtaining
rigidity in a more general context.

%%%%%%%%%%%%%%%%%%%%%%%%%%%%%%%%%%%%%%%%%%%%%%%%%%%%%%%%%%%%%%%%%%%%%%%%
\subsection{Rigid currents on complex surfaces}
%%%%%%%%%%%%%%%%%%%%%%%%%%%%%%%%%%%%%%%%%%%%%%%%%%%%%%%%%%%%%%%%%%%%%%%%

Another simple example of a rigid pseudo-effective class can be constructed as follows.
Assume that $M$ is a projective surface. Then the pseudo-effective cone $\PP_M$
is the dual of the nef cone, see e.g. \cite[Theorem 4.1]{Bo}. Assume that $C\subset M$ is an
irreducible curve with $C^2< 0$. Then its cohomology class $[C]$ spans an extremal ray
of $\PP_M$, see e.g. \cite[Theorem 3.21]{Bo}. The cohomology class $[C]$ is represented by
the current of integration over $C$, and that is the unique closed positive current
representing $[C]$, see \cite[Proposition 3.15]{Bo}. In
this example the class $[C]$
is not nef. However, one can also construct an example of a surface $M$ with a
nef irreducible curve $C\subset M$ such that the class
$[C]$ is rigid, see \cite[Example 1.7]{DPS}.

Proposition \ref{_rigid_on_torus_Proposition_}
can be used to obtain rigid currents on a K3 surface.
Another source of rigid currents originates in S. Cantat's
seminal research from early 2000's, \cite{_Cantat:K3_}, \cite{_Cantat:K3-2_}. 
Cantat has shown that
any positive current of dynamical origin is rigid if it is
obtained from a hyperbolic automorphism of a K3 surface. These results
were extended and generalized at great length by Dinh and Sibony
(\cite{DS3}; see also Subsection \ref{sec_dynam}).

Let $f\colon M \to M$ be a holomorphic automorphism of a complex surface.
Following Cantat, we call $f$ {\em loxodromic} if
its action on $H^{1,1}(M)$ has an eigenvalue $\alpha$ that satisfies $|\alpha|>1$.
Since the action of $f$ on $H^{1,1}(M)$ is isometric with
respect to the intersection form which has signature $(1,h^{1,1}(M)-1)$,
it can have at most one eigenvalue $\alpha$ with $|\alpha|>1$.
Denote the corresponding eigenvector by $\eta$.
The standard dominant eigenvalue argument implies
that the sequence $\frac{(f^*)^n[\omega]}{\alpha^n}$
converges to a vector proportional to $\eta$
for all cohomology classes $[\omega]\in H^{1,1}(M)$
outside of a closed subspace. Applying the same
argument to a general K\"ahler form $\omega$ and using
compactness of the set of positive $(1,1)$-currents
with bounded mass, we obtain that the sequence
$\frac{(f^*)^n\omega}{\alpha^n}$ has a subsequence
which converges to a positive closed $(1,1)$-current.

This current is called {\em a dynamical current}.
It is a result of Cantat \cite{_Cantat:K3_},
that dynamical currents are always rigid,
see section \ref{sec_dynam}.

When $M$ is a complex torus, this argument
leads to the situation described in 
Proposition \ref{_rigid_on_torus_Proposition_}:
the dynamical current is smooth and 
transversally K\"ahler with respect
to a translation invariant holomorphic
foliation. 

Surprisingly, on a K3 surface the
structure of dynamical currents is entirely
different. The following observations were
made by Cantat and Dupont \cite{_Cantat_Dupont_}.

Let $T_f^+$ be the
dynamical current associated with
a loxodromic automorphism $f$ of a K3
surface. Since the action of $f$ on $H^{1,1}(M)$
preserves the intersection form,
the eigenvalues of $f$ go in pairs,
and there exists an eigenvector with
an eigenvalue $\alpha^{-1}$.
Let $T^-_f$ be the dynamical current
associated with $f^{-1}$,
its cohomology class is proportional
to the eigenvector of $f|_{H^{1,1}(M)}$
with eigenvalue $\alpha^{-1}$.

In \cite{_Cantat:K3_,_Cantat:automorphisms_},
Cantat has shown that the currents
$T^+_f$, $T^-_f$ have H\"older continuous
potentials, and therefore the product
$T^+_f\wedge T^-_f$ is a well defined
measure $\mu_f$ on a K3 surface. As shown in \cite{_Cantat:automorphisms_},
this measure is equal to the measure of maximal entropy for $f$.

In \cite{_Cantat_Dupont_}
Cantat and Dupont have shown that
the measure $\mu_f=T^+_f\wedge T^-_f$
is absolutely continuous only if
the K3 surface $M$ is obtained as
a resolution of a finite quotient of a torus
(``of Kummer type''), and the automorphism
$f$ is induced by a loxodromic automorphism
of a torus.
For non-projective K3 surfaces, this theorem
is due to Filip and Tosatti \cite{_Filip_Tosatti:rigidity_}.

Since 2019--2020, many new results on
rigid currents on K3 surfaces were obtained,
due to Cantat, Dujardin, Filip and Tosatti;
some of them inspired by an earlier versions
of this paper.

On a complex surface $M$, a nef cohomology
class $\eta\in H^{1,1}(M)$ is called {\em parabolic}
if its square vanishes.
In \cite[Theorem 6.11]{_Cantat_Dujardin:random_},
Cantat and Dujardin used the random walk on
a complex surface with sufficiently rich
automorphism group to construct an invariant
measure $\mu_\6$ on the set of limit classes
(a posteriori parabolic) in ${\Bbb P}H^{1,1}(M)$.
They prove that $\mu_\6$-almost all
parabolic classes are rigid.

In \cite{_Filip_Tosatti:canonical_} this result
is developed futher. Citing an early
version of this paper, Filip and Tosatti
show that all parabolic currents on a K3 surface $M$
are rigid, if its Picard rank is $\ge 3$
and $M$ does not admit $(-2)$-curves.\footnote{
These assumptions are  needed in order for the automorphism
group to act ergodically on the projectivization of the set
of parabolic classes, identified with the absolute of the
hyperbolic space form obtained as the projectivization
of the K\"ahler cone in the space
$H^{1,1}_\bbR(M)$.} 
Note that in the present paper we managed to
avoid this restriction, using the mapping
class group in place of the automorphism group.

In \cite{_Filip_Tosatti:gaps_}
this direction is developed further:
Filip and Tosatti prove that the support of 
some examples of rigid currents on a K3 surface $M$ 
(called ``canonical currents'' in that paper) is strictly smaller than $M$.

%%%%%%%%%%%%%%%%%%%%%%%%%%%%%%%%%%%%%%%%%%%%%%%%%%%%%%%%%%%%%%%%%%%%%%%%
\subsection{Rigid currents and extremal currents}
%%%%%%%%%%%%%%%%%%%%%%%%%%%%%%%%%%%%%%%%%%%%%%%%%%%%%%%%%%%%%%%%%%%%%%%%

The notion of a rigid current that we study in the present paper is
related to another important notion of complex analytic geometry,
namely to the notion of an {\it extremal positive current}, see \cite{Le} and \cite{De1}.

For a compact K\"ahler manifold $X$ denote:
$$
\mathcal{SP}^{p}(X) = \{T\in \EE'(X)^{p,p}_\bbR\st dT=0,\, T\ge 0\}.
$$
This set is a closed convex cone in the space of $(p,p)$-currents.
A current $T\in\mathcal{SP}^p(X)$ is called {\it extremal} if
$T$ generates an extremal ray of that cone. The latter means that
for any decomposition $T = T_1+T_2$ with $T_1,T_2 \in \mathcal{SP}^p(X)$
the currents $T$, $T_1$ and $T_2$ are proportional. It follows from
the Krein--Milman theorem that $\mathcal{SP}^p(X)$ is the closure
of the convex hull of the subset of extremal currents.

For an irreducible closed subvariety $Z\subset X$, the current $\II_Z$
of integration along $Z$ is extremal, but $\mathcal{SP}^p(X)$
may contain extremal currents that are not of the form $\II_Z$,
see \cite[Th\'eor\`eme 4.1]{De1}, where such currents were constructed on $\bbP^n$.

It is clear that an extremal current
may be non-rigid: consider, for example, the current of integration $\II_D$
along an ample divisor $D\subset X$. Vice versa, a rigid current
need not be extremal: consider, for example, the case when $X$ is
a surface containing two irreducible curves $C_1$ and $C_2$ with
$[C_1]^2<0$, $[C_2]^2<0$ and $[C_1]\cdot[C_2] = 0$. In this case
by \cite[Proposition 3.15]{Bo} any convex combination $t\,\II_{C_1}+(1-t)\II_{C_2}$,
$0< t < 1$ is a rigid current, but this current is clearly not extremal.
However, we make the following simple observation.

\begin{prop}
Assume that $T\in \mathcal{SP}^1(X)$ is a rigid current whose
cohomology class $[T]$ generates an extremal ray of the pseudo-effective
cone $\PP_X$. Then the $T$ is extremal.
\end{prop}
\begin{proof}
Assume that $T = T_1 + T_2$ with $T_1,T_2\in \mathcal{SP}^1(X)$.
Then $[T] = [T_1] + [T_2]$ and therefore $[T] = \alpha_1[T_1] = \alpha_2[T_2]$
for some $\alpha_1,\alpha_2>0$, because $[T]$ generates an extremal ray of $\PP_X$.
It follows from rigidity of $T$ that $T = \alpha_1 T_1 = \alpha_2 T_2$.
\end{proof}

As a consequence of the above proposition, we note that the rigid
parabolic currents that are described in Corollary \ref{cor_main}
are extremal, because the corresponding cohomology classes lie
on the isotropic part of the boundary of the K\"ahler cone, therefore
generating extremal rays of $\PP_X$. 

%%%%%%%%%%%%%%%%%%%%%%%%%%%%%%%%%%%%%%%%%%%%%%%%%%%%%%%%%%%%%%%%%%%%%%%%
\subsection{SYZ conjecture and Lelong numbers}
%%%%%%%%%%%%%%%%%%%%%%%%%%%%%%%%%%%%%%%%%%%%%%%%%%%%%%%%%%%%%%%%%%%%%%%%

Let $X$ be a hyperk\"ahler manifold, that is,
a holomorphically symplectic compact manifold
of K\"ahler type (see section \ref{sec_hk} for
more details). We will tacitly assume
that the  hyperk\"ahler manifolds we consider
are {\em of maximal holonomy} (subsection \ref{sec_HK_intro}),
that is, satisfy $\pi_1(X)=1$,
$H^{2,0}(X)=\C$.

The main motivation for the current research is due
to the following conjecture, called {\em hyperk\"ahler SYZ conjecture};
it was independently proposed by Tyurin, Bogomolov, Huybrechts and Sawon
(see \cite{_Verbitsky:SYZ_} and refe\-rences therein).
Let $\eta$ be an integral parabolic class on a hyperk\"ahler
manifold of maximal holonomy, and $L$ the holomorphic
line bundle that satisfies $c_1(L)=\eta$.
The hyperk\"ahler SYZ conjecture claims that $L$ is semiample. 

As follows from the foundational results of Matsushita
\cite{_Matsushita:fibred_}, the projective map associated
with such $L$ defines a Lagrangian fibration, that is,
a map $X \to S$ with all fibres holomorphic Lagragian 
in $X$.

The SYZ conjecture is proven for all known examples
of hyperk\"ahler manifolds. For the deformations
of Hilbert schemes of K3 surfaces it follows from
\cite{Bayer-Macri}, \cite{_Markman:fibrations_K3^n_},
and for the deformations of generalized Kummer varieties it
follows from \cite{Yoshioka}.
For two sporadic examples, O'Grady 6 and O'Grady 10
and their deformations, the SYZ conjecture is proven
in \cite{Mongardi-Rap}, \cite{Mongardi-Onorati}.

The first idea  was to obtain the Lagrangian
fibration as the kernel of a semi-positive form
representing $\eta$. Then it became clear that
the positive currents representing irrational
parabolic classes are even more interesting.

The singularities of positive currents representing
parabolic classes bear much importance 
to the work on the SYZ conjecture.
One can describe the singularities of positive closed 
$(p,p)$-classes using the Lelong numbers, \cite{Dem}.
These numbers represent ``algebraically significant''
singularities of a current. A positive closed $(1,1)$-current
on a K\"ahler manifold with vanishing Lelong numbers
can always be obtained as a limit of closed positive
$(1,1)$-forms. However, the
converse assertion is false: currents
with non-zero Lelong numbers can also be obtained as
limits of closed, positive $(1,1)$-forms. 

In \cite{_Verbitsky:SYZ_} it was shown
that an integral parabolic class $\eta$
is $\Q$-effective (can be represented
by a $\Q$-divisor) if its Lelong numbers
vanish. In complex dimension 4, vanishing
of Lelong numbers actually implies the SYZ
conjecture, \cite[Theorem 1.6]{_Gongyo_Matsumura_}.

From the earliest results of Cantat it
is clear that dynamical currents on a K3 surface
have vanishing Lelong numbers. His argument
trivally extends to the dynamical currents
on all hyperk\"ahler manifolds. Semicontinuity
of Lelong numbers and Theorem \ref{thm_orbit_closure}
then imply that the Lelong numbers vanish
for all rigid currents associated with 
parabolic classes which are not orthogonal
to rational classes, if the assumptions of
\ref{thm_orbit_closure} hold.

There is a hope to apply this argument
to the SYZ conjecture. Were we able to understand the
dependence of a rigid current on its
cohomology class, we could be able to
show that any parabolic class can be
represented by a positive current with
vanishing Lelong numbers. This would
have proven the SYZ conjecture in complex
dimension 4, and $\Q$-effectivity
of parabolic line bundles in arbitrary
dimension.

%%%%%%%%%%%%%%%%%%%%%%%%%%%%%%%%%%%%%%%%%%%%%%%%%%%%%%%%%%%%%%%%%%%%%%%%
\subsection{History of this paper}
%%%%%%%%%%%%%%%%%%%%%%%%%%%%%%%%%%%%%%%%%%%%%%%%%%%%%%%%%%%%%%%%%%%%%%%%

This paper started in 2016 as a joint project of Nessim Sibony and Misha Verbitsky.
This early project was never finished due to advance of the pandemic
and Nessim's tragic death on October 30, 2021. By that time,
most statements were fleshed out, but the Lie-theoretic
part of the proof remained completely wrong. This old version was
almost entirely rewritten with the help of Andrey Soldatenkov,
who became our third co-author and completed the faulty arguments.

\subsection{Acknowledgements}

A.S. is most grateful to the Instituto Nacional de Matem\'atica Pura e
Aplicada (IMPA) in Rio de Janeiro for its help and hospitality. Both A.S. and M.V.
enjoyed the warm and friendly atmosphere of IMPA during the preparation
of this paper. The authors are grateful to the anonymous referees for
carefully reading the paper, for the useful remarks and suggestions,
and especially for correcting the definition of the diameter of a pseudo-effective class in section 3.

%%%%%%%%%%%%%%%%%%%%%%%%%%%%%%%%%%%%%%%%%%%%%%%%%%%%%%%%%%%%

\section{Statement of the main results}

%%%%%%%%%%%%%%%%%%%%%%%%%%%%%%%%%%%%%%%%%%%%%%%%%%%%%%%%%%%%

In the present paper we study rigidity properties of parabolic cohomology classes
on compact hyperk\"ahler manifolds. We will call a cohomology class $\alpha\in H^2(X,\bbR)$
{\em strongly irrational} if for any non-zero $\beta\in H_2(X,\bbQ)$ we have $\langle\alpha,\beta\rangle\neq 0$.
It is clear that a very general cohomology class is strongly irrational, more precisely,
the complement of strongly irrational classes is a countable union of hyperplanes. 
Our results show that any strongly irrational parabolic
class on a compact hyperk\"ahler manifold is rigid.
%One distinctive feature of our rigit parabolic classes is that they are not
%contained in the subspace of $H^{1,1}_\bbR(X)$ spanned by the Neron-Severi lattice.

We next give a precise statement of our main result. By a {\it compact hyperk\"ahler manifold} we
will mean a compact simply connected K\"ahler manifold $X$ such that $H^0(X,\Omega^2_X)$
is spanned by a holomorphic symplectic form $\sigma$. We will discuss the necessary
facts from the theory of hyperk\"ahler manifolds in section \ref{sec_hk}.
For the statement of the main theorem we only need to recall that there
exists a natural non-degenerate quadratic form $q$ on $H^2(X,\bbQ)$
which is called the Beauville--Bogomolov--Fujiki (BBF) form. A cohomology
class $\alpha\in H^2(X,\bbC)$ satisfies $\alpha^{2k} = 0$ if and only
if $q(\alpha)=0$, where $2k = \dim_\bbC(X)$. For a class $\alpha\in H^2(X,\bbC)$
we will denote by $a^\perp$ the orthogonal complement of $\alpha$ in $H^2(X,\bbC)$
with respect to the BBF form.

\begin{thm}\label{thm_main}
Assume that $X$ is a compact hyperk\"ahler manifold with $b_2(X)\ge 7$
and $u\in H^{1,1}_\bbR(X)$ is a parabolic class, i.e.\ a nef class
with $q(u)=0$, where $q$ is the BBF form. The class $u$ is
rigid in  either of the following cases:
\begin{enumerate}
\item if $u$ is strongly irrational, i.e.\ $u^\perp \cap H^2(X,\bbQ) = 0$,
\item if $u^\perp \cap H^2(X,\bbQ)$ is spanned by a vector $v\in H^{2,0}(X)\oplus H^{0,2}(X)$.
\end{enumerate}
\end{thm}

A proof of the above theorem will be given in section \ref{sec_main_proof}.
The proof uses two main ingredients that, we believe, are interesting on their own. The first ingredient is a semicontinuity
result for diameters of pseudo-effective classes, Theorem \ref{thm_diam}.
This result applies to arbitrary families of K\"ahler manifolds, see section \ref{sec_diam}.
The second ingredient is the parabolic Teichm\"uller space introduced in section \ref{sec_teich}
and a theorem that guarantees that orbits of certain points under the
mapping class group action are dense in the parabolic Teichm\"uller space.
The latter result relies on Ratner's theory.

\begin{rem}
Note that the dynamical currents of 
S. Cantat never satisfy the first of the two
conditions of Theorem \ref{thm_main}. Indeed, for
any automorphism of a projective hyperk\"ahler manifold 
its action on the transcendental lattice has finite order
by \cite{_Oguiso:mcmullen_}.\footnote{The transcendental
lattice in this case is the orthogonal complement in $H^2(X, \Z)$ of the N\'eron--Severi lattice $H^{1,1}(X)\cap H^2(X, \Z)$.}
Therefore, any eigenvector $v\in H^2(X, \R)$ of the
automorphism which has a real eigenvalue $\lambda >1$
belongs to $W\otimes_\Q \R$, where $W = H^{1,1}(M)\cap H^2(X, \Q)$.
This implies that Cantat's dynamical
currents never satisfy the first of the two conditions
of Theorem \ref{thm_main}. The second condition might be satisfied
when the Picard group of $X$ has maximal possible rank.
All in all, the Cantat--Dinh--Sibony rigidity theorem is
(almost) never a special case of Theorem \ref{thm_main}. 
\end{rem}

\begin{rem}
We have no examples of irrational (i.e.\ not proportional to 
an element of $H^2(X,\Q)$) non-rigid parabolic cohomology classes $\eta \in H^{1,1}(X)$
on a maximal holonomy hyperk\"ahler manifold.
It would be interesting to find such an example, or to prove
that any irrational parabolic class is rigid.
\end{rem}

\begin{cor}\label{cor_main}
Assume that $X$ is a hyperk\"ahler manifold with $b_2(X)\ge 7$
and non-maximal Picard group, i.e.\ the rank of $\mathrm{Pic}(X)$ is less than $b_2(X)-2$.
Then there exists a non-empty open subset $U$ of the boundary of the K\"ahler cone $\partial \KK_X$
such that the rigid parabolic classes are dense in $U$.
\end{cor}

A proof of the corollary will be given in section \ref{sec_cor_proof}. Let us recall that
currently the following deformation types of compact hyperk\"ahler manifolds are known (see \cite{Ma}):
the K3 surfaces with $b_2 = 22$,
the $\mathrm{K3}^{[n]}$-type with $b_2 = 23$, the $\mathrm{Kum}^n$-type with $b_2 = 7$, the OG10-type with $b_2=24$
and the OG6-type with $b_2=8$. Hence both Theorem \ref{thm_main} and Corollary \ref{cor_main}
apply to all currently known compact hyperk\"ahler manifolds.

As we have already noticed above, the conditions in the Theorem \ref{thm_main}
are sufficient for rigidity of a parabolic class, but they are certainly
not necessary. We expect that there should exist many ways of weakening these
conditions, but we find it useful not to obscure clarity of the ideas
by sophisticated statements. As an example of a possible generalization,
let us consider the following stronger version of the first
part of Theorem \ref{thm_main}.

\begin{thm}\label{thm_main2}
Assume that $X$ is a hyperk\"ahler manifold with the BBF form $q$.
Let $N\subset H^2(X,\bbQ)$ be a $q$-negative subspace containing no MBM classes. Assume that  $b_2(X) - \dim(N)\ge 7$
and $u\in H^{1,1}_\bbR(X)\cap N^\perp$ is a parabolic class such that $u^\perp\cap H^2(X,\bbQ) = N$.
Then the class $u$ is rigid.
\end{thm}

The proof of the above theorem uses the 
same argument as that of Theorem \ref{thm_main},
see section \ref{sec_main2_proof}.

Another version of our main result is a generalization
of \cite[Theorem 4.3.1]{_Filip_Tosatti:canonical_}, which
was a generalization of a result proven in an earlier
version of this manuscript. The following theorem
was proven in \cite{_Filip_Tosatti:canonical_}
for K3 surfaces; however, its proof is valid for 
all hyperk\"ahler manifolds of maximal holonomy.

Recall that a class $\eta \in H^2(X,\R)$
is called {\em irrational} if it is not proportional
to a rational class.

\begin{thm}\label{_main3_Theorem_}
Let $X$ be a compact complex manifold admitting a hyperk\"ahler
structure of maximal holonomy. Assume that $H^{1,1}(X)$
contains no MBM classes, and let $\eta$ be an irrational
parabolic class on the boundary of the ample cone $\Amp$.
Assume, moreover, that $X$ is projective and the
rank of its N\'eron-Severi lattice is $\geq 3$.
Then $\eta$ is rigid.
\end{thm}

The proof of the above theorem uses the 
same argument as that of Theorem \ref{thm_main},
see section \ref{sec_main3_proof}.

\section{Diameters of pseudo-effective classes}\label{sec_diam}

In this section we will study the diameters of pseudo-effective cohomology classes.
The main result of this section is Theorem \ref{thm_diam} which applies to arbitrary
smooth families of compact K\"ahler manifolds. Throughout this section $\pi\colon \XX\to B$
will denote one such family, i.e.\ a proper submersion of complex manifolds with K\"ahler fibres.

\subsection{The pseudo-effective cone and potentials}

We start by recalling several well known definitions. For a general introduction to the
various notions of pluripotential theory (e.g. the plurisubharmonic functions) that we will use below,
see e.g. \cite{H1} and \cite{GZ}.

Let $X$ be a compact K\"ahler manifold and $\omega\in \Lambda^{1,1}_\bbR X$ a K\"ahler form.
Fix an arbitrary volume form $\Vol\in\Lambda^{n,n}_\bbR X$, not necessarily related to $\omega^n$.
Recall that a function $\varphi\colon X\to\bbR\cup\{-\infty\}$ is called {\em quasi-plurisubharmonic}
or {\em q-psh} if it can locally be written as a sum of a smooth function and
a plurisubharmonic function. The quasi-plurisubharmonic functions are upper semi-continuous
and belong to $L^1(X)$, see \cite{GZ}.  

Recall that a cohomology class $\alpha\in H^{1,1}_{\bbR}(X)$
is called {\em pseudo-effective} if it is represented by a closed positive $(1,1)$-current.
This current is usually not unique, and we introduce an invariant that measures how far
it is from being unique.
Denote by $\PP_X \subset H^{1,1}_{\bbR}(X)$ the set of pseudo-effective classes and
recall that $\PP_X$ is a closed convex cone. 
Let us fix a smooth closed 2-form $\eta\in\Lambda^{1,1}_{\bbR}X$ with $[\eta]=\alpha\in \PP_X$.
By our assumption the class $\alpha$ is represented by a closed positive current $\eta'$.
By the $dd^c$-lemma there exists a real-valued distribution $\varphi\in \DD'(X)$
that satisfies $\eta' = \eta+dd^c\varphi\ge 0$ in the sense of currents.
It is known (see e.g. \cite[Theorem A.5]{HL}) that the distribution $\varphi$ is represented
by a unique quasi-plurisubharmonic function which we will call a {\em potential} of $\alpha$.

We introduce the notation for the space of the normalized potentials:
$$
L^1_0(X) = \left\{ \varphi\colon X\to \bbR\cup\{-\infty\} \,\,\left\vert\,\,\right. \varphi \mbox{ is q-psh,}\, \int_X\varphi\,\Vol = 0\right\}.
$$

Define
\begin{equation}\label{eqn_potentials}
\Phi_\eta = \{\varphi\in L^1_0(X)\st \eta+dd^c \varphi \ge 0 \}.
\end{equation}
It is well known that $\Phi_\eta$ is a compact subset of $L^1(X)$ and that
the functions in $\Phi_\eta$ are uniformly bounded from above, see \cite[Proposition 8.5]{GZ}.
The compactness implies that the subset $\Phi_\eta$ has finite
diameter:
$$
\mathrm{diam}(\Phi_\eta) = \sup_{\varphi,\psi\in \Phi_\eta} \left\{\int_X|\varphi-\psi|\Vol\right\} < +\infty.
$$

The set $\Phi_\eta$ depends on the form $\eta$ representing $\alpha$
in the following way:
if $\eta' = \eta + dd^c\chi$ is another smooth form representing $\alpha$,
then $\Phi_\eta = \Phi_{\eta'} + \chi$. It follows that the diameter depends only on the cohomology
class of $\eta$. We thus obtain a function
$\delta\colon \PP_X\to \bbR$:
\begin{equation}\label{eqn_diam}
\delta(\alpha) = \mathrm{diam}(\Phi_\eta), \mbox{ where } [\eta]=\alpha.
\end{equation}
We note that the function $\delta$ depends on the volume form $\Vol$. However,
this dependence is not of primary importance for us, since we are mostly interested
in rigid classes, i.e.\ in the classes $\alpha$ that have vanishing diameter,
and the fact that the diameter vanishes does not depend on $\Vol$.
For the applications to hyperk\"ahler manifolds, we will introduce in section \ref{sec_volume}
certain canonical volume forms and compute the diameter with respect to them.

\subsection{Potentials in the relative setting}\label{sec_pot_rel}

The definitions above extend to the relative setting.
More precisely, let $\pi\colon \XX\to B$
be a proper submersion of complex manifolds of relative dimension $n$.
Assume that there exists a smooth 2-form $\tilde{\omega}\in\Lambda^{1,1}{\XX}$
on the total space $\XX$ such that for any $t\in B$ the restriction $\omega_t = \tilde{\omega}|_{\XX_t}$
is a K\"ahler form on the fibre $\XX_t = \pi^{-1}(t)$. We fix a fibrewise volume
form $\VVol$ on $\XX$ and denote by % i.e.\ a smooth pointwise non-vanishing section $\VVol\in\Lambda^{n,n}{\XX}$,
$\Vol_t = \VVol|_{\XX_t}$ the volume form on the fibre.

The spaces $H^{1,1}_\bbR(\XX_t)$ form a $C^\infty$-vector bundle over $B$
which we denote by $H^{1,1}_\bbR(\XX/B)$. The pseudo-effective cone $\PP_{\XX/B}$ is now
a subset of the total space of this bundle, and by \cite[Theorem A.1]{AHD} this
subset is closed.

As above, for any $t\in B$ we introduce the space of normalized
potentials $L^1_0(\XX_t)$. Given $\alpha_t\in \PP_{\XX_t} \subset \PP_{\XX/B}$
and $\eta_t \in \Lambda^{1,1}_{\bbR}\XX_t$ with $[\eta_t]=\alpha_t$, we
introduce the set of normalized $\eta_t$-psh functions $\Phi_{\eta_t}\subset L^1_0(\XX_t)$. Computing
the diameter of $\Phi_{\eta_t}$, we obtain the diameter function
$$\delta\colon \PP_{\XX/B}\to \bbR.$$

To study the continuity properties of the function $\delta$, we may
restrict to the case when $B$ is a polydisc in $\bbC^m$ with a basepoint $0\in B$.
By Ehresmann's theorem there exists a $C^\infty$-retraction $\XX\to\XX_0\simeq X$
that gives, combined with the projection to $B$, a $C^\infty$-trivialization $\XX\simeq X\times B$. We fix one such
trivialization and identify the fibres $\XX_t$ with $X$ as $C^\infty$-manifolds.

Let $t_i\in B$ be a sequence of points, $t_i\to 0$ as $i\to \infty$. Given
a sequence of functions $\varphi_i\in L^1(\XX_{t_i})$ and $\varphi\in L^1(\XX_0)$,
we will say that $\varphi_i$ converge to $\varphi_0$ in $L^1(X)$ if
$\int_X|\varphi_i - \varphi_0|\Vol_0 \to 0$, where we use the above identification of
the fibres $\XX_t$ with $X$.

The trivialization $\XX\simeq X\times B$ also gives an identification of the cohomology
groups of the fibres $\XX_t$, hence given a sequence $\alpha_{i}\in\PP_{\XX_{t_i}}$
of pseudo-effective classes we may speak of their convergence to a class $\alpha_0\in H^{1,1}_\bbR(\XX_0)$.
Assume that we have such a convergent sequence $\alpha_i\to \alpha_0$ as $i\to\infty$.
Recall that we have a smooth family $\tilde{\omega}$ of K\"ahler forms
on the fibres $\XX_t$. Using the theory of Kodaira--Spencer \cite{KS}, we represent the cohomology
classes on $\XX_t$ by $\omega_t$-harmonic forms and get a sequence of smooth $(1,1)$-forms
$\eta_{i}$ with $[\eta_{i}] = \alpha_{i}$ and $\eta_{i}\to \eta_0$ in the $C^\infty$-topology.
Let $\varphi_i\in \Phi_{\eta_{i}}$ be a sequence of $\eta_i$-psh potentials of the positive currents representing
the classes $\alpha_i$, i.e.\ $\eta_i + dd^c\varphi_i \ge 0$. The following proposition,
proven in greater generality in \cite{DGG}, extends to the relative setting the property
of uniform boundedness of the potentials.

\begin{prop}\label{prop_bounded}
In the above setting there exists a constant $C\in\bbR$ such that for any $i> 0$
$$
0\le \sup_X(\varphi_i) \le C.
$$
\end{prop}

\begin{proof}
Recall that by the definition of the potentials we have
$$
\int_X \varphi_i \, \Vol_{t_i} = 0.
$$
The claim now follows from \cite[Proposition 3.3]{DGG}.
\end{proof}

\subsection{Semi-continuity of the diameter}

The main result of this section is the following theorem
which was already proven in \cite[Lemma 6.4]{_Cantat_Dujardin:random_}
in the case of a fixed complex structure. Let us mention that
the results similar to ours (in particular to Proposition \ref{prop_potentials} below)
using similar techniques have been obtained by various authors, see
e.g. \cite[Proposition 6.6]{BGL}, \cite[Sections 2 and 3]{DGG} and \cite[Section 2]{PT}.

\begin{thm}\label{thm_diam}
The function $\delta\colon \PP_{\XX/B}\to \bbR$ defined above is upper semi-continuous.
\end{thm}

Before giving a proof we make some preparations. We use the notation
introduced in section \ref{sec_pot_rel}. In particular, we assume that $B$
is a polydisc in $\bbC^m$ with a basepoint $0\in B$, fix a $C^\infty$-trivialization
$\XX\simeq X\times B$ and consider a sequence of points $t_i\in B$
such that $t_i\to 0$. We assume that $\alpha_{i}\in\PP_{\XX_{t_i}}$ is a convergent
sequence of pseudo-effective classes, $\alpha_{i}\to \alpha_0$.
We let $\eta_{i}$ be the smooth forms with $[\eta_{i}] = \alpha_{i}$ and $\eta_{i}\to \eta_0$ in the $C^\infty$-topology.
We let $\varphi_i\in \Phi_{\eta_{i}}$ be a sequence of $\eta_i$-psh potentials.

Since $\pi$ is a holomorphic submersion, every point $x\in \XX$
has an open neighbourhood biholomorphic to the polydisc $\VV = \Delta^n\times \Delta^m$,
where $\Delta$ is the unit disc in $\bbC$ and the map $\pi$ is the projection onto $\Delta^m$.
Since the fibres of $\pi$ are compact, some open neighbourhood of $\XX_0$ in $\XX$ can be covered by finitely
many such open subsets.

We restrict to one of the open subsets $\VV\subset \XX$ as above. We identify all fibres
of $\pi|_\VV$ with $\Delta^n$ and obtain a sequence of $(1,1)$-forms $\eta_i$ and functions
$\varphi_i$ on $\Delta^n$. The following lemma is essentially \cite[Theorem 4.1.9(a)]{H2}
applied to our particular situation.

\begin{lem}\label{lem_potentials}
In the above setting, there are two mutually exclusive possibilities:
\begin{enumerate}
\item the sequence $\varphi_i$
converges to $-\infty$ uniformly on compact subsets of $\Delta^n$;
\item the sequence $\varphi_i$ contains a subsequence that converges in $L^1_{\mathrm{loc}}$.
\end{enumerate}
\end{lem}
\begin{proof}
Fix a constant $0 <r < 1$.
Since the forms $\eta_i$ converge to $\eta_0$, the condition
$\eta_i+ dd^c\varphi_i\ge 0$ implies that for some constant $a>0$ the functions
$\tilde{\varphi}_i(z) = \varphi_i(z) + a|z|^2$ are plurisubharmonic on $\Delta_r^n$ for all $i$,
where $\Delta_r\subset \Delta$ is the disc of radius $r$.
In particular, the functions $\tilde{\varphi_i}$ are subharmonic for the flat metric on $\Delta_r^n$, and
since they are also uniformly bounded from above by Proposition \ref{prop_bounded}, we may apply \cite[Theorem 4.1.9(a)]{H2} to them.
The theorem in loc.\ cit.\ implies that one of the two possibilities from the statement
of the lemma holds for $\tilde{\varphi_i}$ on $\Delta^n_r$. But then it also holds for $\varphi_i$,
since the function $a|z|^2$ is bounded on the polydisc. Since $r$ was arbitrary, the lemma follows. 
\end{proof}

\begin{prop}\label{prop_potentials}
After possibly passing to a subsequence, the functions $\varphi_i$ converge in $L^1(X)$ to some $\varphi_0\in\Phi_{\eta_0}$.
\end{prop}
\begin{proof}
{\em Step 1.} We cover a neighbourhood of $\XX_0$ by finitely many local charts $\VV$ as above.
The convergence of $\varphi_i$ in $L^1_{\mathrm{loc}}(\Delta^n)$ for each of the charts
implies the convergence of $\varphi_i$ in $L^1(X)$. To apply Lemma \ref{lem_potentials}
we are going to check that $\varphi_i$ do not converge uniformly to $-\infty$
on all compact subsets of a chart.

Define an open subset $W\subset \XX_0$ by the following condition: $x\in\XX_0$ lies in $W$
if and only if for some open neighbourhood $U\subset\XX$ of $x$ the sequence $\varphi_i$
converges to $-\infty$ uniformly on $U$. The latter means that for any $c>0$ there exists
$i_0$ such that for all $i\ge i_0$ we have $\varphi_i \le -c$ on $\XX_{t_i}\cap U$.

We need to show that $W=\emptyset$. First we check that $W\neq \XX_0$.
By Proposition \ref{prop_bounded} we have $\sup_X(\varphi_i)\ge 0$ and by the compactness of $X$
there exists a sequence of points $x_i\in\XX_i$ with $\varphi_i(x_i) \ge 0$. If
$x_0\in\XX_0$ is any limit point of the sequence $x_i$ (which exists because $X$ is compact),
then $x_0\notin W$. Assuming that $W\neq \emptyset$ we can find a point $x\in \partial W$ and
a local chart $\VV\subset \XX$ around $x$ that has non-empty intersection both 
with $W$ and $\XX_0\setminus W$.
In this case the restrictions of $\varphi_i$ to $\VV$ can not converge to $-\infty$ uniformly
on compact subsets, because $x\notin W$. By Lemma \ref{lem_potentials}
the restrictions of $\varphi_i$ to $\VV$ contain an $L^1_{\mathrm{loc}}$-convergent
subsequence. But this contradicts the fact that $\varphi_i$
uniformly converge to $-\infty$ in a neighbourhood of any point of $W\cap \VV\neq \emptyset$.
The contradiction shows that $W=\emptyset$ and completes the proof.

{\em Step 2.} We have found a subsequence $\varphi_i$ that converges in $L^1(X)$ to a function $\varphi_0\in L^1(X)$.
The function $\varphi_0$ satisfies the condition $\eta_0+dd^c\varphi_0\ge 0$ by continuity,
and we may assume that $\varphi_0$ is $\eta_0$-psh. To conclude the proof we will show
that $\int_X \varphi_0\,\Vol_0 = 0$, so that $\varphi_0\in\Phi_{\eta_0}$.

We have $\int_X|\varphi_i - \varphi_0|\Vol_0 \to 0$ as $i\to\infty$. Consequently, $\int_X|\varphi_i|\,\Vol_0$
is bounded from above uniformly in $i$ and $\int_X \varphi_i\, \Vol_0\to \int_X \varphi_0\, \Vol_0$.
We have $\Vol_{t_i} = f_i \Vol_0$ for some $C^\infty$-functions $f_i$ uniformly converging to
the constant function equal to $1$ as $i\to \infty$. Therefore
$$
\int_X\varphi_i\,\Vol_0 = \int_X\varphi_i(\Vol_0 - \Vol_{t_i}) = \int_X \varphi_i(1-f_i)\Vol_0.
$$
The latter integral converges to zero, hence $\int_X \varphi_0\, \Vol_0 = 0$ as claimed.
\end{proof}

\begin{rem}
In the above proof we never used the fact that the limit class $\alpha_0$ is pseudo-effective.
In fact, the argument above produces a function $\varphi_0$ that satisfies the
condition $\eta_0+dd^c\varphi_0\ge 0$, hence it provides an alternative proof
of \cite[Theorem A.1]{AHD}.
\end{rem}

\begin{proof}[Proof of Theorem \ref{thm_diam}]
We consider a sequence of pseudo-effective classes $\alpha_i\in\PP_{\XX/B}$ in the fibres
over the points $t_i\in B$, and assume $\alpha_i\to \alpha_0\in\PP_{\XX_0}$.
We need to show that $\delta(\alpha_0)\ge \limsup \delta(\alpha_i)$.
As explained above, we represent the classes $\alpha_i$ by harmonic forms $\eta_i$.
The volume form on the fibre $\Vol_i = \VVol|_{\XX_i}$ is the restriction of the relative volume form.
Passing to a subsequence we assume that the limit $d = \lim \delta(\alpha_i)$
exists and choose the potentials $\varphi_i,\psi_i\in \Phi_{\eta_i}$ so
that
$$
d = \lim_{i\to\infty}\int_X|\varphi_i-\psi_i|\Vol_i.
$$
By Proposition \ref{lem_potentials}, after passing to subsequences, the sequences of potentials converge in $L^1(X)$:
$\varphi_i\to \varphi_0$, $\psi_i\to \psi_0$ for some $\varphi_0,\psi_0\in\Phi_{\eta_0}$.
Since the volume forms also converge $\Vol_i\to \Vol_0$, we get that
$$
d = \int_X|\varphi_0-\psi_0|\Vol_0.
$$
It follows that $d\le \mathrm{diam}(\Phi_{\eta_0})=\delta(\alpha_0)$ and this completes the proof.
\end{proof}

\subsection{Diameters of dynamical classes}\label{sec_dynam}
We recall the results of \cite{_Cantat:K3_} about
the vanishing of diameters of dynamical pseudo-effective classes.

Following Dinh \cite{Di}, we call $\alpha\in \PP_X$ dynamical if there exists an automorphism $f\colon X\to X$
and a volume form $\Vol$ on $X$ such that $f^*\alpha =\lambda \alpha$ with $\lambda > 1$
and $f^*\Vol = \Vol$. We represent $\alpha$
by a smooth $(1,1)$-form $\eta$. Since the map $\lambda^{-1}f^*$ preserves $\alpha$,
we have $\lambda^{-1}f^*\eta = \eta+dd^c u$ for some $u\in C^{\infty}(X)$ with $\int_X u\,\Vol = 0$.
We see that $\lambda^{-1}f^*$ acts on $\Phi_\eta$:
it maps $\varphi\in \Phi_\eta$ to $F(\varphi) = u + \lambda^{-1}f^*\varphi \in \Phi_\eta$.
Since $\lambda^{-1}f^*$ is invertible, $F$ is an automorphism of $\Phi_\eta$.
Recall that $f^*$ preserves $\Vol$ and observe that $F$ is a uniform contraction:
$$
\|F(\varphi)-F(\psi)\|_{L^1} = \frac{1}{\lambda}\int_X f^*|\varphi-\psi|\Vol = \frac{1}{\lambda}\|\varphi-\psi\|_{L^1}.
$$
Since the set $\Phi_\eta$ is compact and non-empty, this is only possible if it consists of one point.
We therefore obtain the following statement.

\begin{thm}[\cite{_Cantat:K3_, DS3}] \label{_dyn_current_rigid_Theorem_}
In the above setting $\delta({\alpha}) = 0$.
\end{thm}

It follows that there exists a unique positive current representing the class $\alpha$.
In this case we call that current {\em rigid}.

\section{Hyperk\"ahler manifolds: basic notions}\label{sec_hk}

Starting from this section we will focus our attention on a special class of complex manifolds:
the hyperk\"ahler manifolds. This section is an exposition of the necessary background material
on hyperk\"ahler manifolds and it does not contain any new results, although the presentation
may be somewhat different from what one can find in the literature. For more details
on hyperk\"ahler manifolds see \cite{Hu} or \cite{So}. We will denote by $\bbH$ the algebra
of quaternions. Multiplication from the left by an imaginary quaternion of unit length
defines a complex structure on the real vector space $\bbH^n$. We will say that an
endomorphism of $\bbH^n$ is quaternionic-linear if it commutes with these complex structures
for all unit length imaginary quaternions. We will denote by $\Sp(n)$ the group
of quaternionic-linear transformations of $\bbH^n$
that preserve the standard quaternionic-Hermitian metric. Throughout this
section we fix a compact simply connected $C^\infty$-manifold $X$ such that $\dim_\bbR X = 4n$.

\subsection{Hyperk\"ahler metrics, symplectic structures and the BBF form}\label{sec_HK_intro}
A Riemannian metric $g$ on $X$ is {\em hyperk\"ahler} if the holonomy group of its Levi--Civita
connection $\nabla^g$ is contained in $\Sp(n)$. The action of $\Sp(n)$ on $\bbH^n$ commutes with
the family of complex structures defined via left multiplication by the imaginary unit quaternions.
Therefore by the holonomy principle the Levi--Civita connection $\nabla^g$ preserves a family
of complex structures on $X$.
The family of complex structures arising this way is parametrized by the points of the sphere $S^2$
and will be called the {\em twistor family} of $g$.

We will always assume that a hyperk\"ahler metric on $X$ is of {\it maximal holonomy}, i.e.\ the holonomy
group of $\nabla^g$ is isomorphic to $\Sp(n)$. In this case the complex structures forming the twistor
family are the only complex structures preserved by $\nabla^g$. In fact, if $X$ admits one hyperk\"ahler
metric of maximal holonomy, it follows from the Beauville--Bogomolov decomposition theorem \cite{Be}
that any hyperk\"ahler metric on $X$ has maximal holonomy, see e.g. \cite[Proposition 3.1]{So}

Given a complex structure $I$ on $X$, we will denote by $X_I$ the corresponding complex manifold.
We will say that the complex structure $I$ is of {\em hyperk\"ahler type} if $X$
admits a hyperk\"ahler metric $g$ such that $\nabla^g I = 0$.
An equivalent condition follows from Yau's solution to the Calabi conjecture
(see e.g. \cite[section 2]{Be}): $I$ is of hyperk\"ahler type if and only if $X_I$ admits
a K\"ahler metric and $H^0(X_I,\Omega^2_{X_I})$ is spanned by a holomorphic symplectic
form. If $I$ is of hyperk\"ahler type, the manifold $X_I$ is usually called irreducible
holomorphic symplectic (IHS) or simply hyperk\"ahler in the literature.
For brevity, we will call $X$ a {\em hyperk\"ahler manifold} if it admits a complex structure
of hyperk\"ahler type, or, equivalently, if $X$ admits a hyperk\"ahler metric of maximal holonomy.

%Recall another equivalent characterization of hyperk\"ahler manifolds, see \cite{SV}.
%A {\it C-symplectic structure} on the real $4n$-dimensional manifold
%$X$ is a complex 2-form $\sigma\in \Lambda^2_\bbC X$ such that $d\sigma = 0$,
%$\sigma^{n+1} =0$ and $\sigma^n\wdg\bar{\sigma}^n$ is pointwise non-vanishing.
%Given a C-symplectic structure $\sigma$ one can easily construct a complex
%structure $I$ such that $\sigma$ is a holomorphic symplectic form for $I$.
%We will say that a C-symplectic structure is of K\"ahler type if $X_I$ admits
%a K\"ahler metric. Actually, in this case $I$ is of hyperk\"ahler type and $X$ is
%hyperk\"ahler. From now on we will assume that $X$ is a hyperk\"ahler manifold of real dimension $4n$.

Given a hyperk\"ahler metric $g$ one can pick a triple of $\nabla^g$-parallel complex structures $I$, $J$ and $K$
such that $K = IJ=-JI$. We will call the tuple $(g,I,J,K)$ a {\it hyperk\"ahler structure} on $X$.
Denote by $\omega_I$, $\omega_J$ and $\omega_K$ the corresponding K\"ahler forms and by $\Vol_g$
the Riemannian volume form. The 2-form $\sigma_I = \omega_J +\ii\omega_K$ is a holomorphic
symplectic form on $X_I$. Recall the following well known relations (see e.g. \cite{Hu}):
$\omega_I^{2n} = \omega_J^{2n} = \omega_K^{2n} = (2n)!\Vol_g$ and $\sigma_I^n\wdg\bar{\sigma}_I^n = (n!)^2 4^n\Vol_g$.
We will usually assume that a hyperk\"ahler metric has unit volume, meaning that $\int_X \Vol_g = 1$.
Equivalently, the symplectic structure $\sigma_I$ has unit volume if $\int_X \sigma_I^n\wdg\bar{\sigma_I}^n = (n!)^{2}4^{n}$.

Next we recall some basic facts about the cohomology of $X$, see e.g. \cite{Be}, \cite{Fu} and \cite{Hu}.
The vector space $H^2(X,\bbQ)$ carries a natural non-degenerate quadratic form $q$ called the Beauville--Bogomolov--Fujiki (BBF)
form. This form may be characterized by the following property: there exists a positive constant $c_X\in \bbQ$
such that for all $a\in H^2(X,\bbQ)$ we have the {\it Fujiki relation}:
\begin{eqnarray}\label{eqn_Fujiki}
\int_X a^{2n} = c_X q(a)^n.
\end{eqnarray}
We will normalize $q$ in such a way that it is primitive and integral on $H^2(X,\bbZ)$.
The form $q$ may be expressed merely in terms of the intersection product and the Pontryagin
classes of $X$, hence it is a topological invariant of $X$ independent of the choice of a complex
structure, see e.g. \cite[section 2]{So}. The BBF form has signature $(3,b_2(X)-3)$.
Assume that $I$ is of hyperk\"ahler type, and, as above, $\omega_I \in \Lambda^{1,1}_I X$ is a K\"ahler form
and $\sigma_I\in H^0(X,\Omega^2_{X_I})$ is the symplectic form, $\sigma_I = \omega_J+\ii\omega_K$.
Consider the subspace $W \subset H^2(X,\bbR)$ spanned by the classes $[\omega_I]$, $[\omega_J]$
and $[\omega_K]$. The quadratic form $q$ is positive on all K\"ahler classes (see e.g. \cite[section 1.9]{Hu}),
therefore $q|_W$ is positive definite.

Note that for a hyperk\"ahler metric $g$
of unit volume we have $\int_X \omega_I^{2n} = (2n)!$, and hence $c_X q(\omega_I)^n = (2n)!$ by the formula (\ref{eqn_Fujiki}).
Since $q(\omega_I)$ is a positive real number, it does not depend on the metric and the complex
structure. So for any K\"ahler form $\omega$ of any unit volume hyperk\"ahler metric we have $q(\omega) = b_X$
for the intrinsically defined constant $b_X = \sqrt[n]{(2n)!c_X^{-1}}$.

\subsection{Teichm\"uller spaces, MBM classes and the K\"ahler cone}\label{sec_Teich}

Let us recall the definitions of various Teichm\"uller-type spaces associated with the
hyperk\"ahler manifold $X$. For a background on Teichm\"uller spaces and the related Kuranishi
theory of local deformations of complex structures see \cite{Ko}, \cite{Me} and \cite{V1}.

We denote by $\Diff(X)$ the Lie group of diffeomorphisms of $X$ with the Fr\'echet topology (see \cite[Section 3]{Me}) and by $\Diff^\circ(X)$ the
connected component of the identity in this group. It follows from the work of Hitchin and 
Sawon \cite{HS} that all diffeomorphisms of $X$ are orientation-preserving
and all complex structures of hyperk\"ahler type induce the same
orientation on $X$, see \cite[section 2.2]{So} for a detailed explanation of this fact. The quotient $\MC(X) = \Diff(X)/\Diff^\circ(X)$ is
called the {\em mapping class group} of $X$.

We will consider the following spaces equipped with their natural Fr\'echet topologies
and the corresponding quotients by the action of $\Diff^\circ(X)$. 

\begin{enumerate}
\item $\Comp(X)$, the space of all complex structures of hyperk\"ahler type on $X$, and the Teichm\"uller space $\Teich(X) = \Comp(X)/\Diff^\circ(X);$
\item $\HK(X)$, the space of all hyperk\"ahler metrics of unit volume on $X$, and $\TeichHK(X) = \HK(X)/\Diff^\circ(X);$
\item $\HKa(X)$, the space of pairs $(g,I)$, where $g\in \HK(X)$ and $I$ is a complex structure with $\nabla^g I = 0$,
and $\TeichHKa(X) = \HKa(X)/\Diff^\circ(X);$
\item $\HKb(X)$, the space of hyperk\"ahler structures on
  $X$, that is, tuples $(g,I,J,K)$, where $g\in \HK(X)$, and $I$, $J$, $K$ are complex structures
with \[ \nabla^g I = \nabla^g J = \nabla^g K = 0
\text{\ \ and \ \ }K = IJ = -JI;
\] the corresponding Teichm\"uller space
$\TeichHKb(X) = \HKb(X)/\Diff^\circ(X);$
%\item $\CSymp(X)$, the space of K\"ahler type C-symplectic structures of unit volume,
%and $$\TeichCS(X) = \CSymp(X)/\Diff^\circ(X).$$
\end{enumerate}

We will slightly abuse the notation, and the image of $I\in \Comp(X)$ in $\Teich(X)$ will also
be denoted by $I$. We will also omit $X$ from the notation and simply write $\Teich$, $\TeichHK$, and so on.
All of the Teichm\"uller spaces defined above admit natural period maps with
values in certain homogeneous domains that we now describe.
Denote $V = H^2(X,\bbQ)$, $d = b_2(X) = \dim(V)$ and recall that $q$ is a non-degenerate quadratic form of signature $(3,d-3)$.
We will say that a subspace $W\subset V_\bbR$ is positive if $q|_W$ is positive definite.
We list the period domains and describe the corresponding period maps.
\begin{enumerate}

\item The period domain for $\Teich$:
\begin{eqnarray*}
\Dom &=& \{L\subset V_\bbR\st \dim(L) = 2,\,L\mbox{ is oriented and positive}\}\\
&\simeq& \rmO(3,d-3)/\SO(2)\times \rmO(1,d-3).
\end{eqnarray*}
The period map $\Per\colon\Teich\to \Dom$ sends a complex structure $I$ to the oriented subspace
$L = \langle \mathrm{Re}[\sigma_I],\mathrm{Im}[\sigma_I]\rangle$, where $\sigma_I$ is a
holomorphic symplectic form on $X_I$, uniquely defined up to the action of $\bbC^*$,
and $[\sigma_I]$ is its cohomology class. Note that $\Diff^\circ$ acts trivially on $H^\sdot(X,\bbQ)$,
so the period map is well defined.

It is sometimes convenient to use an alternative description of $\Dom$ as an open subset
of a quadric in $\bbP(V_\bbC)$, namely:
$$
\Dom \simeq \{x\in \bbP(V_\bbC) \st q(x) = 0,\,\, q(x,\bar{x})>0\},
$$
where $x$ corresponds to the oriented subspace $L = \langle\mathrm{Re}(x),\mathrm{Im}(x)\rangle \subset V_\bbR$
in the definition of $\Dom$, see \cite[\S 2.4]{V1}.

\item The period domain for $\TeichHK$:
\begin{eqnarray*}
\DomHK &=& \{W\subset V_\bbR\st \dim(W) = 3,\,W\mbox{ is oriented and positive}\}\\
&\simeq& \rmO(3,d-3)/\SO(3)\times \rmO(d-3).
\end{eqnarray*}
The period map $\PerHK\colon\TeichHK\to \DomHK$ sends a hyperk\"ahler metric $g$ to the subspace
$\langle [\omega_I], [\omega_J], [\omega_K]\rangle$ spanned by the classes of the K\"ahler forms of all
$\nabla^g$-parallel complex structures. Note that this subspace has an intrinsic
orientation induced by the multiplication in $\bbH$: for a pair of non-proportional
imaginary quaternions $a,b\in\bbH$ the triple $a,b,\mathrm{Im}(ab)$ forms an oriented basis
in the space of imaginary quaternions.

\item The period domain for $\TeichHKa$:
\begin{eqnarray*}
\DomHKa &=& \{(W,u) \st W\in \DomHK,\, u\in W,\, q(u) = b_X\}\\
&\simeq& \rmO(3,d-3)/\SO(2)\times \rmO(d-3),
\end{eqnarray*}
where the constant $b_X$ is defined in the end of section \ref{sec_HK_intro}.
The period map $\PerHKa\colon\TeichHKa\to \DomHKa$ sends a pair $(g,I)$ to the pair $(W,u)$, where
$W = \langle [\omega_I], [\omega_J], [\omega_K]\rangle$ and $u = [\omega_I]$.

\item The period domain for $\TeichHKb$:
\begin{eqnarray*}
\DomHKb &=& \{(u,v,w) \st u,v,w\in V_\bbR,\, q(u) = q(v) = q(w) = b_X,\nonumber\\ && \qquad q(u,v) = q(v,w) = q(w,u) = 0\}\\
&\simeq& \rmO(3,d-3)/\rmO(d-3).
\end{eqnarray*}
The period map \[ \PerHKb\colon\TeichHKb\to \DomHKb\]
sends a hyperk\"ahler structure $(g,I,J,K)$ to
the triple $( [\omega_I], [\omega_J],
[\omega_K] )
\in \DomHKb$. 

%\item The period domain for $\TeichCS$:
%\begin{eqnarray*}
%\DomCS &=& \{(u,v) \st u,v\in V_\bbR,\,\ q(u) = q(v) = b_X,\,\, q(u,v) = 0\}\\
%&\simeq& \rmO(3,d-3)/\rmO(1,d-3).
%\end{eqnarray*}
%The period map $\PerCS\colon\TeichCS\to \DomCS$ sends a C-symplectic structure $\sigma$ to
%the pair $(\mathrm{Re}[\sigma],\mathrm{Im}[\sigma])$.
%
\end{enumerate}

We have a number of natural maps between the Teichm\"uller spaces
and period domains. We summarize them in the following commutative diagram.

\begin{equation}\label{eqn_spaces}
\begin{tikzcd}[]
\TeichHK \dar{\PerHK} & \TeichHKa \dar{\PerHKa}\drar{\varphi}\arrow[l,"\,\,\tau_1"']\rar{\tau_2\,} & \Teich \dar{\Per}\\% & \TeichCS \dar{\PerCS}\arrow[l, "\,\,\tau_3"'] \\
\DomHK & \DomHKa \lar\rar & \Dom% & \DomCS \lar
\end{tikzcd}
\end{equation}

The maps $\tau_i$ are defined as follows: $\tau_1(g,I) = g$, $\tau_2(g,I) = I$.
%, and $\tau_3(\sigma) = I_\sigma$, where $I_\sigma$ is the complex structure determined by the C-symplectic form $\sigma$, see \cite{SV}.
The corresponding maps between the period domains are induced by the obvious inclusions of Lie groups.
%We also have $\TeichHKb\simeq \TeichHKa\times_\Teich \TeichCS$ and $\DomHKb \simeq \DomHKa\times_\Dom \DomCS$,
%the isomorphisms being compatible with the period map $\PerHKb\colon \TeichHKb\to\DomHKb$.

Note that $\tau_1$ is an $S^2$-bundle whose fibres are twistor families of complex structures.
Given $I\in \Teich$, it follows from Aubin--Yau's solution to the Calabi conjecture (see e.g. \cite{GZ} for
a modern treatment and related results)
that any K\"ahler class on $X_I$ can be represented by the K\"ahler form of a unique Ricci-flat
K\"ahler metric. This metric is hyperk\"ahler, as follows from the Beauville--Bogomolov decomposition theorem
\cite[Theorem 1]{Be}, the fact that $X_I$ is simply connected and $H^0(X_I,\Omega^2_{X_I})$ is spanned
by a symplectic form. Therefore the fibre of $\tau_2$ over a point corresponding to the complex structure $I$ is isomorphic
to the {\it K\"ahler cone} of $X_I$ which will be denoted $\KK_{X_I}$.
%The map $\tau_3$ is an $S^1$-bundle whose fibre over $I$ is the set of holomorphic symplectic forms of unit volume on $X_I$.

The central result about Teichm\"uller spaces of hyperk\"ahler manifolds is the global Torelli
theorem \cite{V1} which we now recall. Let us fix a connected component $\Teich^\circ$ of the Teichm\"uller space.
It is clear from (\ref{eqn_spaces}) that the choice of $\Teich^\circ$ uniquely determines connected
components $\TeichHK^\circ$, $\TeichHKa^\circ$ and $\TeichHKb^\circ$.% and $\TeichCS^\circ$.
The Teichm\"uller space $\Teich$ has a structure of non-Hausdorff complex manifold locally isomorphic
to $\Dom$. It is possible to identify non-separated points of $\Teich$ and produce a Hausdorff
complex manifold $\widetilde{\Teich}$ which is mapped to $\Dom$ by the period map $\tilde{\Per}$.
The global Torelli theorem \cite[Theorem 4.29]{V1} claims that $\tilde{\Per}^\circ\colon\widetilde{\Teich}^\circ\to\Dom$
is an isomorphism of complex manifolds.

One can also describe the fibres of the period map $\Per^\circ\colon \Teich^\circ\to \Dom$. Fix an oriented subspace $L\subset V_\bbR$
corresponding to the point $[L] = p\in \Dom$. Recall that if $\rho^\circ(I) = L$, then $L=\langle{\rm Re}(\sigma_I), {\rm Im}(\sigma_I)\rangle$. The restriction of $q$ to $V^{1,1}_\bbR = L^\perp\subset V_\bbR$
has signature $(1,d-3)$ and we denote by $\CC^+\subset V^{1,1}_\bbR$ the {\it positive cone}, i.e.\ one
of the two connected components of the set of $x\in V^{1,1}_\bbR$
such that $q(x)>0$. The connected component is determined by the condition that it contains the K\"ahler classes
of the hyperk\"ahler manifolds in the fibre ${\Per^\circ}^{-1}(p)$. 
The points of the fibre ${\Per^\circ}^{-1}(p)$ correspond to bimeromorphic hyperk\"ahler manifolds, see \cite[Theorem 4.3]{Hu}. These points
are distinguished by the K\"ahler cones of the corresponding
manifolds: each point $[I] \in{\Per^\circ}^{-1}(p)$ may be uniquely identified by specifying
$x\in \CC^+$ that corresponds to a K\"ahler class on $X_I$. It follows that
the period maps $\PerHK^\circ\colon \TeichHK^\circ \to\DomHK$ and $\PerHKa^\circ\colon \TeichHKa^\circ\to \DomHKa$
are open embeddings. One may also describe ${\Per^\circ}^{-1}(p)$ as follows.
Consider the composition $\varphi^\circ = \rho\circ \tau_2^\circ\colon \TeichHKa^\circ\to \Dom$.
Then $\KK = {\varphi^\circ}^{-1}(p)$ can be identified with an open subset of $\CC^+$ and ${\Per^\circ}^{-1}(p)$ is
the set of connected components of $\KK$. 

The subset $\KK\subset \CC^+$ is the complement to a union of hyperplanes in $V^{1,1}_\bbR$
defined as orthogonal complements to certain integral cohomology classes called {\it MBM classes}. Let us recall
the definition of these classes, see \cite{AV1}. For a point $[g]\in \TeichHK^\circ$
consider the corresponding twistor family $\tau_1^{-1}[g]$ and it image $C_g = \varphi(\tau_1^{-1}[g])$
in the period domain $\Dom$. If we identify $\Dom$ with an open subset of a quadric, then
$C_g\subset \Dom$ is a conic which we will call a {\it twistor conic}. Consider all primitive cohomology
classes $x\in H^2(X,\bbZ)$ such that $q(x) < 0$ and $\bbP(x^\perp) \subset \bbP(V_\bbC)$ does
not contain any twistor conic. Such classes are called MBM classes and they admit a number
of equivalent definitions that we will not use, see \cite{AV1}. The set of MBM classes may
depend on the choice of a connected component of the Teichm\"uller space. We will therefore
denote this set $\MBM^\circ$. The central result about MBM classes obtained in \cite[Corollary 5.2]{AV4}
states that there exists a constant $M>0$, possibly depending on the connected component $\Teich^\circ$,
such that for all $x\in\MBM^\circ$ we have $-M\le q(x)<0$. We will briefly rephrase
this by saying that the BBF squares of the MBM classes are bounded (note that we assume the MBM
classes to be primitive by definition).

We can now describe the subset $\KK\subset \CC^+$ defined above more precisely. Consider
the set $\MBM^{1,1} = \MBM^\circ\cap V^{1,1}_\bbR$. For any $x\in \MBM^{1,1}$ let
us denote by $H_x$ the orthogonal complement of $x$ in $V^{1,1}_\bbR$. Then
we have $\KK = \CC^+\setminus \cup_{x\in\MBM^{1,1}} H_x$. Since the BBF squares of MBM classes
are bounded, the collection of hyperplanes $H_x$ is locally finite in $\CC^+$, i.e.\ any
point of $\CC^+$ has a neighbourhood intersecting only a finite number of the hyperplanes $H_x$,
see Proposition \ref{prop_loc_fin}.
The hyperplanes therefore cut the positive cone $\CC^+$ into open chambers, and for
any complex structure $I$ with period $p$ the K\"ahler cone $\KK_{X_I}$ is one of those chambers,
see \cite[\S 5.2]{Ma}.

\subsection{Canonical volume forms on hyperk\"ahler manifolds}\label{sec_volume}

In order to study diameters of pseudo-effective classes, we need to introduce volume
forms on our manifolds. Let us observe that a complex structure of hyperk\"aler type $I$
admits a natural volume form $\Vol_I = \sigma_I^n\wdg\bar{\sigma}_I^n$, where $\sigma_I$
is a symplectic form of unit volume. The form $\sigma_I$ is unique up to multiplication
by $z\in \bbC$ with $|z|=1$, so the form $\Vol_I$ is uniquely defined. We will call $\Vol_I$
the {\it canonical volume form}. The mapping class group action preserves the canonical
volume forms, i.e.\ for any $\varphi\in\Diff(X)$ we have $\varphi^*\Vol_I = \Vol_{\varphi^*I}$.
This follows from the fact that $\varphi^*$ maps a symplectic form into a symplectic form
and preserves the total volume of $X$.

\begin{prop}\label{prop_diam_invar}
Assume that $\alpha\in\PP_{X_I}$ is a pseudo-effective class, $\varphi\in\Diff(X)$
and $I' = \varphi^*I$.
Then $\varphi^*\alpha\in\PP_{X_{I'}}$ and $\delta(\alpha) = \delta(\varphi^*\alpha)$,
where $\delta$ is the diameter function (\ref{eqn_diam}) computed using the canonical volume
form introduced above.
\end{prop}

\begin{proof}
Let $\eta\in \Lambda^{1,1}_\bbR X_I$ be a two-form representing $\alpha$. Then it
is clear from the definition of a potential (\ref{eqn_potentials}) and from the $\varphi^*$-invariance of the
canonical volume form that $\Phi_{\varphi^*\eta} = \varphi^*\Phi_\eta$.
The claim now follows from the definition of the diameter.
\end{proof}

\section{The parabolic Teichm\"uller space and orbits of the mapping class group action}\label{sec_teich}

\subsection{The mapping class group action and density of its orbits}

We start this section by recalling the main results of \cite{V2}. We will use the notation
introduced in the previous section, in particular we will work with a fixed connected component
$\Teich^\circ$ of the Teichm\"uller space. The group $\MC(X)$ acts on the Teichm\"uller spaces
introduced above and permutes their connected components. We denote by $\MC^\circ$ the
subgroup that preserves $\Teich^\circ$. The group $\MC^\circ$ also acts on $\TeichHK^\circ$,
$\TeichHKa^\circ$ and $\TeichHKb^\circ$.% and $\TeichCS^\circ$.
Consider the action of $\MC^\circ$ on the cohomology
of $X$, in particular on the vector space $V$. It follows from \cite{V1}, see also \cite{Ma}, that the image of
$\MC^\circ$ in $\rmO(V,q)$ is a finite index subgroup in $\rmO(H^2(X,\bbZ),q)$. This subgroup,
called the {\it monodromy group}, will be denoted by $\Gamma$.

Let us consider a point of $I\in \Teich$ with the symplectic form $\sigma_I$.
Assume that the real subspace $L = \langle\mathrm{Re}[\sigma_I],\mathrm{Im}[\sigma_I]\rangle \subset V_\bbR$
does not contain non-zero rational vectors. Then it follows from \cite{V2} that the 
orbit $\MC^\circ\cdot I$ is dense in $\Teich^\circ$. The proof consists of
two steps. First one shows that the $\Gamma$-orbit of $\Per(\sigma_I)$ is dense in $\Dom$
using Ratner's theory. Then one uses an additional argument to deal with non-separated points
in $\Teich^\circ$.
% Since $\tau_3$ is an $S^1$-bundle, one deduces from this that the point $\tau_3(\sigma)\in \Teich^\circ$
%has a dense $\MC^\circ$-orbit.
It is also possible to give a rather precise description
of the closures of other $\MC^\circ$-orbits in $\Teich^\circ$, but we will not do this,
because we will not be able to fully generalize this description to the case of the parabolic
Teichm\"uller space.  

Let us recall the statement of Ratner's theorem about orbit closures in the form convenient for us.
For a $\bbQ$-vector space $V$ let  $\mathbf{G}\subset SL(V)$ be a connected semisimple algebraic group
defined over $\bbQ$. The group of $\bbR$-points of $\mathbf{G}$, denoted $\mathbf{G}(\bbR)$,
carries a natural structure of a real Lie group. Let $G = \mathbf{G}(\bbR)^\circ$ be the identity component
of this Lie group. Let $\Gamma\subset \mathbf{G}(\bbQ)\cap G$ be an arithmetic lattice and $S\subset G$
be a Lie subgroup generated by its one-parameter unipotent subgroups.

\begin{thm}[Ratner \cite{Ra}; see also {\cite[Proposition 3.3.7]{KSS}}]\label{thm_Ratner}
In the above setting for any $g\in G$ we have $\overline{\Gamma g S} = \Gamma H g$,
where $H=\mathbf{H}(\bbR)^\circ$ and $\mathbf{H}\subset \mathbf{G}$ is the smallest
algebraic $\bbQ$-subgroup such that $H$ contains $gSg^{-1}$. 
\end{thm}

\begin{cor}\label{cor_Ratner}
In the setting of the above theorem, consider the homogeneous space $G/S$
and its point $p = gS$. Assume that the smallest algebraic $\bbQ$-subgroup of $G$ containing $gSg^{-1}$
is $G$ itself. Then the $\Gamma$-orbit of $p$ is dense in $G/S$.
\end{cor}

\subsection{The parabolic Teichm\"uller space and its period domain}\label{sec_def_domains}

The parabolic Teichm\"uller space parametrizes complex structures of hyperk\"ahler
type on $X$ together with a parabolic cohomology class. As before, we denote
$V = H^2(X,\bbQ)$. For a complex structure of hyperk\"ahler type $I$ denote
by $V^{p,q}_I \subset V_\bbC$ the corresponding Hodge components and
$V^{1,1}_{I,\bbR} = V^{1,1}_I\cap V_\bbR$. Recall that $\alpha\in V^{1,1}_{I,\bbR}$
is called parabolic for $I$ if $q(\alpha) = 0$ and $\alpha$ is nef, i.e.\ $\alpha\in\overline{\KK}_{X_I}$.
Since $\Diff^\circ(X)$ acts trivially on $V$, the notion of a parabolic class depends only
on the isotopy class of $I$, i.e.\ on the point of $\Teich$ corresponding to $I$.

Consider the topological space $\Teich\times V_\bbR$ with the natural action of the
mapping class group $\MC$.
We define the {\it parabolic Teichm\"uller space} as follows:
$$
\TeichP(X) = \{ (I,\alpha)\in \Teich\times V_\bbR \st 0\neq \alpha\mbox{ is parabolic for }I \}.
$$
It is clear that the action of $\MC$ preserves $\TeichP$.
%Note that the subspace $\TeichP$ is neither open nor closed in $\Teich\times V_\bbR$ and it has no natural structure of a manifold;
We consider $\TeichP$ merely as a topological space and do not introduce a manifold structure on it. It clearly admits a natural
map $\tau_p\colon \TeichP\to \Teich$ whose fibre is the set of non-zero parabolic classes for a given point in $\Teich$.
For the fixed connected component $\Teich^\circ$ we will also denote
$$
\TeichP^\circ = \TeichP\cap (\Teich^\circ \times V_\bbR).
$$
It is clear that $\MC^\circ$ acts on $\TeichP^\circ$.

Let us now introduce the {\it parabolic period domain}:
$$
\DomP = \{ (L,u)\in \Dom\times V_\bbR \st 0\neq u\in L^\perp,\, q(u) = 0\}.
$$
There is a natural projection $\DomP\to \Dom$ whose fibres are the non-zero isotropic elements
of type $(1,1)$ for a given point in $\Dom$. The period map $\Per\colon \Teich\to\Dom$
clearly extends to the parabolic period map $\PerP\colon \TeichP\to \DomP$.

Given a subspace $N\subset V$ such that $q|_N$ is negative definite, let us consider
the subdomain $\DD_N = \{ L \in \DD\st L\perp N\}$ in the period domain. The points of $\DD_N$
correspond to the periods of those complex structures on $X$ whose N\'eron--Severi group contains
the fixed negative lattice $N_\bbZ=N\cap H^2(X,\bbZ)$. We will denote by $\DD_{N,p}$, $\Teich_{N,p}$
and $\Teich_{N,p}^\circ$ the preimages of $\DD_N$ in $\DD_p$, $\Teich_p$ and $\Teich_p^\circ$ respectively. 

\subsection{Lie algebras and the parabolic period domain}\label{sec_LieDomP}

To study the $\Gamma$-orbits in the parabolic period domain $\DomP$
we will need some preparations. In this subsection we will
study several homogeneous spaces related to $\DomP$.
The contents of this subsection are purely Lie-theoretic,
and we will combine them with Ratner's theory in subsection \ref{sec_orbitsDomP}.

We will still denote by $V$ a $d$-dimensional vector space over $\bbQ$
together with a non-degenerate quadratic form $q\in S^2V^*$,
but the signature of $q$ will be $(\ell+1,d-\ell-1)$, where $\ell = 0$, $1$ or $2$. 
As before, we denote $V_\bbR = V\otimes_\bbQ \bbR$ and $V_\bbC = V\otimes_\bbQ \bbC$.
We will consider homogeneous spaces that parametrize pairs $(L,u)$, where
$L\subset V_\bbR$ is a positive subspace, $\dim(L)=\ell$, and $0\neq u\in L^\perp$ is an isotropic vector.
The signature of $q|_{L^\perp}$ is $(1,d-\ell-1)$.
In the case $\ell = 0$ we simply consider the space of isotropic vectors in $V_\bbR$.
The case $\ell = 2$ corresponds to the parabolic period domain $\DomP$.

We will denote by $\SO^\circ(V_\bbR,q)$ the identity connected component of the orthogonal Lie group.
Given a positive subspace $L\subset V_\bbR$ and an isotropic vector $0\neq u\in L^\perp$ as above,
consider the homogeneous space
\begin{eqnarray*}
\HH_\ell &=& \SO^\circ(V_\bbR,q)/\SO(L)\times \SO^\circ(L^\perp)^u\\
 &\simeq& \SO^\circ(\ell+1,d-\ell-1)/\SO(\ell)\times \SO^\circ(1,d-\ell-1)^u,
\end{eqnarray*}
where $\SO^\circ(L^\perp)^u \simeq \SO^\circ(1,d-\ell-1)^u$ is the stabilizer of $u$ in $\SO^\circ(L^\perp)\simeq \SO^\circ(1,d-\ell-1)$.
Note that $\HH_2\simeq \DomP$.

Denote by $\frso(L^\perp)^u$ the Lie algebra of the group $\SO^\circ(L^\perp)^u$.
This is the Lie subalgebra of $\frso(L^\perp)$ that annihilates $u$.
This subalgebra acts on $u^\perp/\bbR u$, and we have a short exact sequence of Lie algebras
$$
0\to \frr(L,u)\to \frso(L^\perp)^u \to \frso(u^\perp/\bbR u)\to 0,
$$
where $\frr(L,u)$ is the radical of $\frso(L^\perp)^u$.
Recall that the Lie algebra $\frso(V)$ may be identified with $\Lambda^2 V$
via the quadratic form $q$: a decomposable bivector $v\wdg w$ corresponds
to the element of $\frso(V)$ given by the endomorphism $x\mapsto q(x,w)v - q(x,v)w$.
The algebra $\frr(L,u)$ is abelian and under the above identification of $\frso(V)$ with $\Lambda^2 V$
the subalgebra $\frr(L,u)$ corresponds to the subspace $u\wdg (L+ \bbR u)^\perp$.

Denote by $R(L,u)\subset \SO^\circ(L^\perp)^u$ the unipotent subgroup with
the Lie algebra $\frr(L,u)$. Note that we can not apply Theorem \ref{thm_Ratner}
directly to the homogeneous space $\HH_\ell$, since the stabilizer subgroup
$\SO(L)\times \SO^\circ(L^\perp)^u$ is not generated by unipotents.
So we have to pass to a bigger homogeneous space
\begin{eqnarray}
\widetilde{\HH}_\ell = \SO^\circ(V_\bbR,q)/R(L,u)\label{eqn_Hbig}
\end{eqnarray}
that admits a proper map $\widetilde{\HH}_\ell \to \HH_\ell$.
In order to apply Ratner's theory (Corollary \ref{cor_Ratner} of Theorem \ref{thm_Ratner}) we
set $\mathbf{G} = \SO(V,q)$, $S = R(L,u)$, $\Gamma\subset \mathbf{G}(\bbQ)$ an
arithmetic lattice, and $\mathbf{H}\subset \mathbf{G}$ the smallest algebraic
$\bbQ$-subgroup such that $S\subset \mathbf{H}(\bbR)^\circ$. We have
to show that $\mathbf{H} = \mathbf{G}$, and for this we would like to describe
the minimal $\bbQ$-Lie subalgebra $\frh\subset \frso(V,q)$
such that $\frh_\bbR$ contains $\frr(L,u)$, for a generic pair $(L,u)$.

\begin{prop}\label{prop_algebra}
Assume that $u^\perp$ does not contain non-zero rational vectors.
Denote $\dim(V) = d$ and $\dim(L) = \ell$.
Let $\frh\subset \frso(V,q)$ be a $\bbQ$-Lie subalgebra 
such that $\frh_\bbR$ contains $\frr(L,u)$, where one of the following
conditions is satisfied:
\begin{enumerate}
\item $\ell = 0$ and $d\ge 5$.
\item $\ell = 1$ and $d\ge 6$,
\item $\ell = 2$ and $d\ge 7$,
\end{enumerate}
Then $\frh = \frso(V,q)$.
\end{prop}

The proof of this proposition is given below, see \ref{prop_proof}.
Before proving it we need to establish a number of auxiliary
results. As before, we identify $\frso(V,q)$ with $\Lambda^2V$, and
under this identification the Lie bracket of two decomposable vectors
can be expressed as follows:
$$
[a\wdg b,c\wdg d] = q(b,c)a\wdg d - q(a,c)b\wdg d + q(a,d)b\wdg c - q(b,d)a\wdg c.
$$
We will often tacitly use this expression below.
We will need to work with the complexification $V_\bbC$ and 
consider complex subalgebras $\frr(L,u) = u\wdg (L+ \bbC u)^\perp \subset \frso(V_\bbC)$
in one of the following settings:
\begin{equation}\label{eqn_set1}
L\subset V_\bbC,\quad\! \dim(L) = 1,\quad\! 0\neq u\in L^\perp,\quad\! \dim(L+\bbC u) = 2,\quad\! q(u) = 0,
\end{equation}
or
\begin{equation}\label{eqn_set2}
L\subset V_\bbC,\quad\!\! \dim(L) = 2,\quad\!\! 0\neq u\in L^\perp,\quad\!\! q|_L \mbox{ non-degenerate,} \quad\!\! q(u) = 0.
\end{equation}

\begin{lem}\label{lem_bracket}
Assume that $\frh_\bbC\subset\frso(V_\bbC)$ is a Lie subalgebra
that contains $u_1\wdg W_1$ and $u_2\wdg W_2$ for some non-zero isotropic vectors $u_1$, $u_2\in V_\bbC$
and subspaces $W_1\subset u_1^\perp$, $W_2\subset u_2^\perp$.
Assume that there exists $v\in W_1\cap W_2$ with $q(v) \neq 0$. Then $u_1\wdg u_2\in \frh_\bbC$.
\end{lem}
\begin{proof}
The Lie bracket $[u_1\wdg v, u_2\wdg v]$ is a non-zero multiple of $u_1\wdg u_2$.
\end{proof}

Next we consider separately the three cases $\dim(L) = 0$, $1$ or $2$. We work
under the assumptions of Proposition \ref{prop_algebra}.

\subsubsection{The case $\dim(L) = 0$}

\begin{lem}\label{lem_dim0}
Assume that $d\ge 5$ and $\frh_\bbC\subset\frso(V_\bbC)$ is a Lie subalgebra
that contains $u_1\wdg W_1$ and $u_2\wdg W_2$ for some non-zero isotropic vectors $u_1$, $u_2\in V_\bbC$
and hyperplanes $W_1 = u_1^\perp$, $W_2 = u_2^\perp$. Then $u_1\wdg u_2\in \frh_\bbC$.
\end{lem}
\begin{proof}
The codimension of $W_1\cap W_2$ in $V_\bbC$ is at most 2. The assumption $d\ge 5$ implies
that the subspace $W_1\cap W_2$ is not isotropic. We conclude by applying Lemma \ref{lem_bracket}. 
\end{proof}

\subsubsection{The case $\dim(L) = 1$}

\begin{lem}\label{lem_dim1}
Assume that $d\ge 7$ and that $\frh_\bbC\subset\frso(V_\bbC)$ is a Lie subalgebra
that contains $\frr(L_1,u_1)$, $\frr(L_2,u_2)$ for some $(L_1,u_1)$, $(L_2,u_2)$
as in (\ref{eqn_set1}). Then $u_1\wdg u_2\in \frh_\bbC$.
\end{lem}
\begin{proof}
Let $W_1 = (L_1 + \bbC u_1)^\perp$ and $W_2 = (L_2 + \bbC u_2)^\perp$.
If $u_1\in W_1\cap W_2$, then $u_2\wdg u_1\in \frr(L_2,u_2)\subset \frh_\bbC$
by the definition of $\frr(L_2,u_2)$, which completes the proof in this case.

Assume now that $u_1\notin W_1\cap W_2$. By Lemma \ref{lem_bracket} it suffices
to check that $W_1\cap W_2$ is not an isotropic subspace of $V_\bbC$.
We argue by contradiction and assume that $W_1\cap W_2$ is isotropic.
Note that $u_1 \perp W_1\cap W_2$. It follows that the subspace $Z = \bbC u_1 + W_1\cap W_2$
is also isotropic. Under our assumptions $W_1$, $W_2$ are of codimension $2$ and $Z$ has codimension
not bigger than $3$. This contradicts the assumption that $d\ge 7$ and
shows that $W_1\cap W_2$ is not isotropic.
\end{proof}

Under an additional assumption we can slightly strengthen the above lemma.

\begin{lem}\label{lem_dim1_strong}
Assume that $d\ge 6$ and that $\frh_\bbC\subset\frso(V_\bbC)$ is a Lie subalgebra
that contains $\frr(L_1,u_1)$, $\frr(L_2,u_2)$ for some $(L_1,u_1)$, $(L_2,u_2)$
as in (\ref{eqn_set1}). Assume moreover that $q|_{L_1}$ is non-zero. Then $u_1\wdg u_2\in \frh_\bbC$.
\end{lem}
\begin{proof} The proof is similar to that of Lemma \ref{lem_dim1}. We only need to consider
the case $d = 6$. We have to check that $W_1\cap W_2$ is not isotropic, assuming that $u_1\notin W_1\cap W_2$.
Note that $W_1\cap W_2\subset L_1^\perp$, and under our assumptions $q|_{L_1^\perp}$ is
non-degenerate. We have $\dim(L_1^\perp) = 5$ and $u_1\in L_1^\perp$. If $W_1\cap W_2$ were isotropic,
then $Z = \bbC u_1 + W_1\cap W_2$ would be an isotropic subspace of $L_1^\perp$ of dimension at least 3,
and this is impossible.
\end{proof}

\subsubsection{The case $\dim(L) = 2$}

\begin{lem}\label{lem_dim2}
Assume that $d\ge 7$ and that $\frh_\bbC\subset\frso(V_\bbC)$ is a Lie subalgebra
that contains $\frr(L_1,u_1)$, $\frr(L_2,u_2)$ for some $(L_1,u_1)$, $(L_2,u_2)$
as in (\ref{eqn_set2}).
Denote $W_1 = (L_1 + \bbC u_1)^\perp$ and $W_2 = (L_2 + \bbC u_2)^\perp$.
Then one of the following possibilities holds.
\begin{enumerate}
\item either $u_1\wdg u_2\in \frh_\bbC$;
%\item or $\exists v\in u_1^\perp$, $v\notin W_1$ such that $u_1\wdg v\in \frg_\bbC$;
\item or $\exists v\in u_2^\perp$, $v\notin W_2$ such that $u_2\wdg v\in \frh_\bbC$
\end{enumerate}
\end{lem}
\begin{proof}
We will have to consider a number of cases. We note that the codimension of
the subspace $W_1\cap W_2\subset V_\bbC$ is at most 6.

\medskip
{\it Case 1: $\mathrm{codim}(W_1\cap W_2) \le 5$.}

Since $L_1$ is two-dimensional, it contains an isotropic vector $\tilde{u}_1$.
Therefore there exists a two-dimensional isotropic subspace $H_1 = \langle\tilde{u}_1,u_1\rangle\subset (L_1+\bbC u_1)$.
Clearly $H_1\perp W_1$.
First consider the case when the intersection $H_1\cap W_1\cap W_2$ is non-zero,
and denote by $x$ some non-zero vector from this intersection.
Then since $W_1 = (L_1+ \bbC u_1)^\perp$, we have $x\in (L_1+ \bbC u_1)\cap (L_1+ \bbC u_1)^\perp = \bbC u_1$.
It follows that $u_1\in W_2$ and since $u_2\wdg W_2 = \frr(L_2,u_2)$,
we have $u_2\wdg u_1 \in \frr(L_2,u_2)\subset \frh_\bbC$.

So it remains to consider the case when $H_1\cap W_1\cap W_2 = 0$.
We claim that in this case $W_1\cap W_2$ contains a non-isotropic vector.
If not, then $(W_1\cap W_2)\oplus H_1$ is an isotropic subspace in $V_\bbC$.
But this subspace has codimension at most 3, which is strictly less than $d/2$,
hence the subspace cannot be isotropic. Let $v\in W_1\cap W_2$ be a non-isotropic vector.
We apply Lemma \ref{lem_bracket} to conclude that $u_1\wdg u_2\in \frh_\bbC$.

\medskip
{\it Case 2: $\mathrm{codim}(W_1\cap W_2) = 6$.}
We will consider separately two subcases.

\medskip
{\it Case 2a: $q(u_1, u_2)\neq 0$.}

If $u_1\in W_2$ then $u_2\wdg u_1\in u_2\wdg W_2 = \frr(L_2,u_2)\subset \frh_\bbC$. So we may assume
$u_1\notin W_2$, hence $u_1\notin W_2\cap u_1^\perp$.
We claim that the subspace $W_2\cap u_1^\perp$ cannot be isotropic.
If it were, then $(W_2\cap u_1^\perp)+\bbC u_1$ would also be isotropic.
But the latter subspace has codimension at most 3, contradicting the assumption $d\ge 7$.
We conclude that it is possible to find $x\in W_2\cap u_1^\perp$ with $q(x)\neq 0$.

As the next step, we find $y\in (W_1\cap u_2^\perp \cap x^\perp) \setminus W_2$.
To show that it is possible, we need to check that $W_1\cap u_2^\perp \cap x^\perp\nsubseteq W_2$.
Indeed, otherwise we would have
\begin{displaymath}
\mathrm{codim}_{V_\bbC}(W_1\cap W_2) \le \mathrm{codim}_{V_\bbC}(W_1) + \mathrm{codim}_{W_1}(W_1\cap u_2^\perp \cap x^\perp)\le 5,
\end{displaymath}
contradicting our standing assumption in the case we are considering.

We have shown that $\frh_\bbC$ contains the elements $u_1\wdg y$ and $u_2\wdg x$.
Observe that the Lie bracket $[u_1\wdg y, u_2\wdg x]$ is a non-zero multiple of $x\wdg y$.
Next, consider the Lie bracket $[u_2\wdg x,x\wdg y]$. It is a non-zero multiple
of $u_2\wdg y$. Setting $v = y$, we conclude the proof in this case, because
$v$ satisfies the conditions from case (2) in the statement of the lemma.

\medskip
{\it Case 2b: $q(u_1,u_2) = 0$.}

If $u_1\in W_2$, then $u_1\wdg u_2 \in \frr(L_2, u_2)$ and we conclude that $u_1\wdg u_2\in \frh_\bbC$.
So we may assume that $u_1\notin W_2$ and hence $W_2$ is a hyperplane in $W_2+ \bbC u_1$.

We note that $W_2\subset u_1^\perp$ implies that $W_1\cap W_2 = W_2\cap L_1^\perp$, hence
$\mathrm{codim}(W_1\cap W_2)\le 5$, contradicting our assumptions. It follows that
we can find $x\in W_2\setminus (W_2\cap u_1^\perp)$.

Next observe that it is possible to find $y\in (W_1\cap u_2^\perp)\setminus (W_2+ \bbC u_1)$.
Indeed, $W_2$ is a hyperplane in $W_2+ \bbC u_1$. If $W_1\cap u_2^\perp\subset W_2+ \bbC u_1$,
then we intersect $W_2$ and $W_1\cap u_2^\perp$ inside $W_2+ \bbC u_1$, and we see that $W_1 \cap W_2\cap u_2^\perp$ is of
codimension at most one in $W_1\cap u_2^\perp$. Since $W_1$ has codimension 3 in $V_\bbC$,
the codimension of $W_1\cap u_2^\perp$ in $V_\bbC$ is at most 4, hence $\mathrm{codim}_{V_\bbC}(W_1\cap W_2)\le 5$,
contradicting our assumptions.

By construction, $x\in W_2$ and $y\in W_1$, so we obtain two elements $u_1\wdg y$ and $u_2\wdg x$ in $\frh_\bbC$.
Finally, we compute the Lie bracket $[u_1\wdg y, u_2\wdg x] = q(u_1,x)y\wdg u_2 -q(y,x)u_1\wdg u_2
= u_2\wdg(q(y,x)u_1 - q(u_1,x)y)$. Let $v = q(y,x)u_1 - q(u_1,x)y$.
By construction, $q(u_1,x)\neq 0$. Since $y\notin W_2+ \bbC u_1$ we see that $v\notin W_2$.
We see that $v$ satisfies the conditions of case (2) in the statement of
the lemma, and this completes the proof.
\end{proof}

\subsubsection{Proof of Proposition \ref{prop_algebra}}\label{prop_proof}
%\subsubsection{} 
We are now ready to proceed with the proof of the main proposition.

\begin{proof}[Proof of Proposition \ref{prop_algebra}]%\label{prop_proof}

It is enough to prove that $\frh_\bbC = \frso(V_\bbC)$. For $\xi\in \Aut(\bbC/\bbQ)$
consider the pair $(L^\xi, u^\xi)$, where the superscript $\xi$ denotes
the action of $\xi$ on elements and subspaces of $V_\bbC$. Since $\frh_\bbC$
is defined over $\bbQ$, we have $\frr(L^\xi,u^\xi)\subset \frh_\bbC$.

We claim that the elements of the form $u^\xi$ for $\xi\in \Aut(\bbC/\bbQ)$
span $V_\bbC$. Assuming the contrary, denote by $W\subsetneq V_\bbC$ their span.
The subspace $W$ is stable under the action of $\Aut(\bbC/\bbQ)$, hence it is defined
over $\bbQ$. It follows that $W^\perp$ contains non-zero rational vectors, but
this contradicts our assumptions, because $u\in W$.

Now we can choose a basis of $V_\bbC$ consisting of elements of the form
$u^\xi$ for $\xi\in \Aut(\bbC/\bbQ)$. Let $u_1,\ldots u_d$ be such a basis
and $L_1,\ldots L_d$ the corresponding subspaces.

In the case $\dim(L) = 0$ and $d\ge 5$ we conclude the proof by applying Lemma \ref{lem_dim0}:
it follows from that lemma that $\frh_\bbC$ contains all elements of the form
$u_i\wdg u_j$. But such elements generate $\Lambda^2 V_\bbC$, so $\frh_\bbC = \frso(V_\bbC)$.
Analogously, in the case $\dim(L) = 1$ and $d\ge 6$ we conclude by applying Lemma \ref{lem_dim1_strong}.

Finally, in the case $\dim(L) = 2$ and $d\ge 7$ we apply Lemma \ref{lem_dim2} and observe that there
are two possibilities. Either for all $i,j=1,\ldots d$ we have $u_i\wdg u_j\in \frh_\bbC$,
and then we conclude as above. Or there is a pair of indices, say $i=1$, $j=2$, such that
the second half of the conclusion of Lemma \ref{lem_dim2} holds. In this case 
we have the following: $u_2\wdg (W_2+\bbC v)\subset \frh_\bbC$, where $v\in u_2^\perp$,
$v\notin W_2$, and $W_2 = (L_2+ \bbC u_2)^\perp$. Let us consider the
subspaces $W_2' = W_2 + \bbC v$ and $L_2' = L_2 \cap (W_2')^\perp = L_2\cap v^\perp$. Then $\dim(L_2') = 1$,
$W_2' = (L_2' + \bbC u_2)^\perp$ and $\frh_\bbC$ contains the subalgebra $\frr(L_2', u_2)$.
We have thus reduced to the case $\dim(L) = 1$. Note that $L_2'$ could be isotropic,
but since we assume that $d\ge 7$, we may conclude by applying Lemma \ref{lem_dim1}.
\end{proof}

\subsection{Orbits in the parabolic period domain}\label{sec_orbitsDomP}

In this subsection we will use Theorem \ref{thm_Ratner} to study $\Gamma$-orbits
of certain points in the parabolic period domain $\DomP$.
We will rely on Proposition \ref{prop_algebra} proven above.
We again assume that $V = H^2(X,\bbQ)$
for a hyperk\"ahler manifold $X$ and $q\in S^2V^*$ is the BBF form. We have $d =\dim(V) = b_2(X)$.
The group $\Gamma\subset \rmO(H^2(X,\bbZ),q)\cap \SO^\circ(V_\bbR,q)$ is an arithmetic lattice. 

\begin{cor}\label{cor_period1}
Let $(L,u)\in \DomP$. Assume that $d\ge 7$ and $u^\perp$ does not contain non-zero rational vectors.
Then the $\Gamma$-orbit of $(L,u)$ is dense in $\DomP$.
\end{cor}
\begin{proof}
We note that $\DomP\simeq \HH_2$. We consider the homogeneous
space $\widetilde{\HH}_2$ introduced in equation (\ref{eqn_Hbig})
from section \ref{sec_LieDomP} and
apply Corollary \ref{cor_Ratner} to it. We
set $\mathbf{G} = \SO(V,q)$, $g=1$, $S = R(L,u)$ and $\mathbf{H}\subset \mathbf{G}$
the smallest algebraic $\bbQ$-subgroup such that $S\subset \mathbf{H}(\bbR)^\circ$.
Proposition \ref{prop_algebra} implies that $\mathbf{H} = \mathbf{G}$,
so all the assumptions of Corollary \ref{cor_Ratner} are satisfied
and the $\Gamma$-orbit of $(L,u)$ is dense in $\widetilde{\HH}_2$.
Hence the image of this orbit under the map $\widetilde{\HH}_2\to \HH_2$
is dense in $\HH_2\simeq \DomP$.
\end{proof}

More generally, let us consider a $q$-negative subspace $N\subset V$ and the
corresponding period domain $\DD_{N,p}$, see section \ref{sec_def_domains} for
the definition. Let us denote by $\Gamma_N\subset \Gamma$ the stabilizer of $N$.
The above corollary admits the following generalization,
whose proof is obtained by passing to the orthogonal complement to $N$ in $V$.

\begin{cor}\label{cor_period11}
Let $(L,u)\in \DD_{N,p}$. Assume that $\dim(V) - \dim(N)\ge 7$ and $u^\perp\cap H^2(X,\bbQ) = N$.
Then the $\Gamma_N$-orbit of $(L,u)$ is dense in $\DD_{N,p}$.
\end{cor}

We will also consider non-generic orbits of one particular type that we now describe.
Fix $v\in V$ with $q(v)>0$ and consider the subvariety
$$
\DomP^v = \{(L,u)\in \DomP\st v\in L\}\subset \DomP.
$$
Note that $\DomP^v\simeq \SO^\circ(2,d-2)/\SO^\circ(1,d-2)^u\simeq \HH_1$.

\begin{cor}\label{cor_period2}
Let $(L,u)\in \DomP^v$. Assume that $d\ge 7$ and $\langle u,v\rangle^\perp$ does not contain non-zero rational vectors.
Then the closure of the $\Gamma$-orbit of $(L,u)$ contains $\DomP^v$.
\end{cor}
\begin{proof}
We consider the orthogonal complement $V'=v^\perp \subset V$ and let $L'=L\cap V'$.
Fixing the base point $(L,u)\in \DomP^v$ we see that $\DomP^v$ is isomorphic to the
homogeneous space $\HH_1$ for the subgroup $\SO(V',q)\subset \SO(V,q)$.
We have $\DomP^v\simeq \HH_1 \subset \DomP\simeq \HH_2$ and consider
the corresponding inclusion $\widetilde{\HH}_1\subset \widetilde{\HH}_2$,
see equation (\ref{eqn_Hbig}) from section \ref{sec_LieDomP}.

Let $\mathbf{G} = \SO(V,q)$, $g=1$, $S = R(L,u)$ and $\mathbf{H}\subset \mathbf{G}$
the smallest algebraic $\bbQ$-subgroup such that $S\subset \mathbf{H}(\bbR)^\circ$.
Since $\SO(V',q)$ is a $\bbQ$-subgroup containing $S$, it is clear that $\mathbf{H}\subset \SO(V',q)$.
We apply Proposition \ref{prop_algebra} to $V'$ and $(L',u)$. The assumptions
are satisfied because $d' = \dim(V')\ge 6$, $\dim(L')=1$ and $u^\perp\cap V'$
does not contain non-zero rational vectors. The proposition implies
that $\mathbf{H} = \SO(V',q)$, and by Theorem \ref{thm_Ratner} the closure of the
image of $\Gamma$ in $\widetilde{\HH}_2$ contains $\widetilde{\HH}_1$.
This implies that the closure of the $\Gamma$-orbit of $(L,u)$ in $\DomP$
contains the image of $\widetilde{\HH}_1$ under the map $\widetilde{\HH}_2\to \HH_2\simeq \DomP$
The latter image is $\HH_1\simeq \DomP^v$.
\end{proof}

%%%%%%%%%%%%%%%%%%%%%%%%%%%%%%%%%%%%%%%%%%%%%%%%%%%%%%%%%%%%%%%%%%%%%%%%
\subsection{Orbits in the parabolic Teichm\"uller space}
\label{_orbits_parabolic_Subsection_}
%%%%%%%%%%%%%%%%%%%%%%%%%%%%%%%%%%%%%%%%%%%%%%%%%%%%%%%%%%%%%%%%%%%%%%%%

Now we will use the results of the previous section to show that certain orbits
of the mapping class group are dense in the parabolic Teichm\"uller space.
Recall that we have the period map $\PerP^\circ\colon \TeichP^\circ\to\DomP$.
We first need to relate the closures of some orbits in $\TeichP^\circ$
and the closures of their images in the period domain $\DomP$.
We recall that $V = H^2(X,\bbQ)$ and $q\in S^2V^*$ is the BBF form.

\begin{lem}\label{lem_chamber}
Let $(I,u)\in \TeichP^\circ$ and $(L,u) = \PerP(I,u)$. Assume that
the biggest rational subspace contained in $u^\perp\subset V_\bbR$ has dimension either zero or one,
and in the latter case it is spanned by a $q$-positive vector.
Then $u$ lies in the closure of exactly one chamber of the positive cone,
namely in the closure of the K\"ahler cone of $X_I$.
\end{lem}

\begin{proof}
Let $H = V^{1,1}_{I,\bbR} = L^\perp\subset V_\bbR$ and $H_0 = V^{1,1}_{I,\bbQ} \otimes_\bbQ \bbR$, where $V^{1,1}_{I,\bbQ} = V^{1,1}_{I,\bbR} \cap V$
is the $\bbQ$-subspace in $V^{1,1}_{I,\bbR}$ spanned by the N\'eron-Severi group of $X_I$.

It follows from our assumptions that $H_0\neq H$: otherwise $H = L^\perp$ would be a subspace
of $V_\bbR$ defined over $\bbQ$, and so would be $L$, hence $u^\perp$ would contain a two-dimensional
rational subspace. Also observe that $u\notin H_0$: otherwise $u^\perp$ would contain
a rational subspace $H_0^\perp$ of dimension at least three. 

Recall that we have a set of MBM classes $\MBM_I^{1,1}$ that define a collection
of hyperplanes $x^\perp\subset H$, $x\in \MBM_I^{1,1}$. Since the MBM classes are integral,
we have $\MBM_I^{1,1}\subset H_0$. It is shown below in Corollary \ref{cor_loc_fin} that
the collection of hyperplanes defined by the MBM classes is locally finite around $u$,
i.e.\ there exists an open subset $U\subset H$, $u\in U$ such that $U$ intersects
only a finite number of hyperplanes $x^\perp$.

The hyperplanes $x^\perp$, $x\in \MBM_I^{1,1}$ cut the positive
cone $\CC^+$ into open chambers, and it follows that $U$ meets only a finite number
of those chambers. Also note that $u$ does not lie in any of the hyperplanes $x^\perp$ where $x\in \MBM_I^{1,1}$,
because the MBM classes are rational and negative.
It follows that $u$ lies in the closure of exactly
one chamber that has to be the K\"ahler cone of $X_I$, because $u$ is by definition a nef class.
\end{proof}

\begin{prop}\label{prop_density}
Assume that $(I,u)\in \TeichP^\circ$ and let $(L,u) = \PerP(I,u)$. Denote
by $T\subset \TeichP^\circ$ the closure of the $\MC^\circ$-orbit of $(I,u)$
and by $D\subset \DomP$ the closure of the $\Gamma$-orbit of $(L,u)$.
Assume that one of the following holds:
\begin{enumerate}
\item either $u^\perp$ does not contain non-zero rational vectors,
\item or there exists a non-zero rational vector $v\in L$ such that $\langle u,v\rangle^\perp$
does not contain non-zero rational vectors. 
\end{enumerate}
Then $T = (\PerP^{\circ})^{-1}(D)$. 
\end{prop}
\begin{proof}
We always have $T \subset (\PerP^{\circ})^{-1}(D)$ and we only need to check the opposite inclusion.
Our assumptions on $u$ imply that a $\bbQ$-subspace contained in $u^\perp$
has dimension at most one. In the case 1 this is clear. Assume that we are in the case 2
and $W\subset u^\perp$ is a $\bbQ$-subspace of dimension greater than one.
Then $W\cap v^\perp$ would be a non-trivial $\bbQ$-subspace
in $\langle u, v \rangle^\perp$, contradicting our assumptions. In fact, it is
clear that in the case 2 the biggest rational subspace of $u^\perp$ is spanned by $v$.
Hence we may apply Lemma \ref{lem_chamber} proven above.

Recall that the points of the fibre of the period map $\Per^\circ\colon \Teich^\circ\to\Dom$ over $L$
are in bijection with the chambers of the positive cone $\CC^+$. Since $u$ is in the closure
of exactly one chamber by Lemma \ref{lem_chamber}, the fibre of the parabolic period map $\PerP^\circ$ over $(L,u)$ consists of
exactly one point $(I,u)$. Therefore the $\MC^\circ$-orbit of $(I,u)$ consists of Hausdorff
points of $\TeichP$. Now it is clear that $T = (\PerP^{\circ})^{-1}(D)$.
%density of the orbit of $(L,u)$ in $\DomP$ implies density
%of the orbit of $(I,u)$ in $\TeichP^\circ$.
\end{proof}

Recall the following well known fact about arrangements of hyperplanes in a hyperbolic space.
In the proof of Proposition \ref{prop_density} we have used a version of it, Corollary \ref{cor_loc_fin} below.

\begin{prop}\label{prop_loc_fin}
Let $H$ be a vector space over $\bbR$ of dimension $d+1$, where $d \ge 1$, and
$q\in S^2H^*$ a quadratic form of signature $(1,d)$. Assume that $A\subset H$
is a discrete subset that satisfies the following condition: there exists a constant
$M > 0$ such that for any $x\in A$ we have $q(x) \ge -M$.
Then the collection of hyperplanes $x^\perp$ is locally finite in the positive cone
of $H$. More precisely, if $h\in H$ is positive, i.e.\ $q(h)>0$, then there exists
an open neighbourhood $U\subset H$, $h\in U$ such that $U\cap x^\perp\neq\emptyset$
only for a finite number of $x\in A$.
\end{prop}
\begin{proof}
Clearly, we may assume that $q(h)=1$. Let $U$ be the set of linear
combinations $t h + w$, where $t\in \bbR$ with $1/2 < t < 2$ and $w\in h^\perp$ with $|q(w)| < 1/8$.
Assume that for some $x\in A$ we have $U\cap x^\perp\neq \emptyset$. Then
$x = sh+v$ for some $s\in\bbR$ and $v\in h^\perp$. Since $U\cap x^\perp\neq \emptyset$, there exists $y=th+w\in U$
with $w\in h^\perp$ such that $q(x,y) = 0$. The latter condition implies
that $st + q(v,w) = 0$ and therefore $|st| \le|q(v)|^{1/2}|q(w)|^{1/2}$. 
From the definition of $U$ we see that $1/|t| < 2$ and $|q(w)| < 1/8$,
implying that $s^2 < |q(v)|/2$.

Next observe that $q(x) = s^2 + q(v)$ and $q(v) \le 0$ since $q(v,h) = 0$ and the signature
of $H$ is $(1,d)$. The condition $q(x) \ge -M$ implies that $|q(v)| \le M + s^2$,
and therefore, $|q(v)| < M + |q(v)|/2$. This implies that $|q(v)|< 2M$
and $s^2 < M$. We conclude that $x$ lies in a compact subset of $H$,
and since $A$ is discrete, there can be only a finite number of such $x\in A$.
%Choosing an orthonormal basis in $h^\perp$
%we identify $(H,q)$ and $\bbR^{d+1}$ with the quadratic form $q(x)=x_0^2-x_1^2-\ldots-x_d^2$.
%Then $h = (1,0,\ldots,0)$. Let us consider a neighbourhood of the form
%$$
%U =\{y=(y_0,\ldots,y_d)\in\bbR^{d+1}\st |1 - y_0| < \varepsilon, |y_1| < \varepsilon,\ldots,|y_d|<\varepsilon\},
%$$
%where $0 < \varepsilon < 1$ is to be chosen later. Assume that $x\in A$ and $y\in U\cap x^\perp$.
%Then $x_0y_0 = \sum_{j=1}^d x_j y_j$, and it follows from the definition of $U$ that $$|x_0| = \left|\sum_{j=1}^d \frac{y_j}{y_0}x_j\right|
%\le \frac{\varepsilon}{1-\varepsilon}\sum_{j=1}^d|x_j|.$$
%This implies that
%$$
%x_0^2\le \frac{\varepsilon^2 d}{(1-\varepsilon)^2}\sum_{j=1}^d x_j^2,
%$$
%and
%$$q(x) = x_0^2 - \sum_{j=1}^d x_j^2 \le - a \sum_{j=1}^d x_j^2,$$ where $a = 1 - \frac{\varepsilon^2 d}{(1-\varepsilon)^2}$.
%We choose $\varepsilon$ so that $a>0$. By our assumptions on $A$, we have $-M \le q(x)$, and combining this
%with  the previous inequalities we obtain
%$$\sum_{j=1}^d x_j^2 \le \frac{M}{a}\quad\mbox{ and }\quad x_0^2\le \frac{M(1-a)}{a}.$$
\end{proof}

\begin{cor}\label{cor_loc_fin}
Assume that $H$, $q$ and $A$ are as in Proposition \ref{prop_loc_fin}. Assume moreover
that there exists a subspace $H_0\subsetneq H$ such that $A\subset H_0$.
Let $u\in H\setminus H_0$ be such that $q(u) = 0$. Then there exists
an open neighbourhood $U\subset H$, $u\in U$ such that $U\cap x^\perp\neq\emptyset$
only for a finite number of $x\in A$.
\end{cor}
\begin{proof}
The assumption $u\notin H_0$ implies that there exists $z\in H_0^\perp$ with $q(u,z)\neq 0$.
For $h = u + tz$ with $t\in \bbR$ we have $q(h) = 2tq(u,z) + t^2q(z)$ and we can
find such $t$ that $q(h)>0$. Let $U'$ be a neighbourhood of $h$ obtained by applying
Proposition \ref{prop_loc_fin} to $h$. Note that for any $x\in A$ the hyperplane $x^\perp$
contains $tz$, so $y\in x^\perp$ if and only if $y + tz\in x^\perp$. Hence the open subset
$U = U' - tz$ is the required neighbourhood of $u$.
\end{proof}

We finally prove the main result about density of the orbits in the parabolic Teichm\"uller space.
We recall our notation: $X$ is a hyperk\"ahler manifold, $d = b_2(X)$, $\TeichP^\circ$ is a fixed
component of the parabolic Teichm\"uller space, $\MC^\circ$ is the subgroup of the mapping
class group that acts on $\TeichP^\circ$.

\begin{thm}\label{thm_density}
Assume that $d\ge 7$, $(I,u)\in \TeichP^\circ$ and $u^\perp$ does not contain non-zero rational vectors. Then
the $\MC^\circ$-orbit of $(I,u)$ is dense in $\TeichP^\circ$.
\end{thm}
\begin{proof}
According to Proposition \ref{prop_density} it is enough to check that the orbit 
of $\PerP(I,u)$ is dense in $\DomP$. This follows from Corollary \ref{cor_period1}.
\end{proof}

Now we consider the subdomain $\DD_{N,p}\subset\DomP$ for a $q$-negative subspace $N\subset H^2(X,\bbQ)$
and the corresponding subspace $\Teich_{N,p}$ of the parabolic Teichm\"uller space, see section \ref{sec_def_domains} for
the definitions. We denote by $\MC^\circ_N\subset \MC^\circ$ the stabilizer of $N$ in the mapping class group.

\begin{thm}\label{thm_density2}
Let $(I,u)\in \Teich_{N,p}^\circ$.
Assume that $b_2(X) - \dim(N) \ge 7$ and $u^\perp\cap H^2(X,\bbQ) = N$.
Moreover, assume that $N$ does not contain any MBM classes of $X$.
Then the $\MC_N^\circ$-orbit of $(I,u)$ is dense in $\Teich_{N,p}^\circ$.
\end{thm}
\begin{proof}
Analogous to the proof of Theorem \ref{thm_density}, considering
the orthogonal complement of $N$ inside $H^2(X,\bbQ)$. Note that
by our assumption $u^\perp$ does not contain any MBM classes,
so the analogue of Proposition \ref{prop_density} may be applied. 
\end{proof}

We also consider the case when $u^\perp$ contains one-dimensional positive rational subspace
spanned by a vector $v\in H^{2,0}(X_I)\oplus H^{0,2}(X_I)$. We recall that
$\DomP^v = \{(L,u)\in \DomP\st v\in L\}\subset \DomP$.
In this case we obtain the following statement.

\begin{thm}\label{thm_orbit_closure}
Assume that $d\ge 7$ and for $(I,u)\in \TeichP^\circ$ there exists a non-zero vector
$v\in (H^{2,0}(X_I)\oplus H^{0,2}(X_I))\cap H^2(X_I,\bbQ)$ such that
$\langle u,v\rangle^\perp$ does not contain non-zero rational vectors.
Then the closure of the $\MC^\circ$-orbit of $(I,u)$ contains ${\PerP^{\circ}}^{-1}(\DomP^v)$.
\end{thm}
\begin{proof}
According to Proposition \ref{prop_density} it is enough to check that the closure of the orbit 
of $\PerP(I,u)$ contains $\DomP^v$. This follows from Corollary \ref{cor_period2}.
\end{proof}

\section{Proofs}

\subsection{Proof of Theorem \ref{thm_main}}\label{sec_main_proof}

{\em Case 1.} We first assume that $u^\perp$ does not contain non-zero rational vectors.
Let us denote by $I$ the complex structure on our hyperk\"ahler manifold and
fix the connected component $\Teich^\circ$ of the Teichm\"uller space containing $I$. It follows from \cite[Theorem 1.1]{AV2}
that there exists $I_0\in \Teich^\circ$ such that the Picard group of $X_{I_0}$ has rank two
and $X_{I_0}$ admits a hyperbolic automorphism $\gamma$. By the construction in \cite{AV2},
the K\"ahler cone of $X_{I_0}$ coincides with the positive cone, and $\gamma$ admits
an eigenvector $u_0\in H^{1,1}_{I,\bbR}(X)$ with eigenvalue $\lambda > 1$, such that
$q(u_0) = 0$. It follows that the class $u_0$ may be chosen nef, hence parabolic.
Since the action of $\gamma$ preserves the canonical volume form on $X_{I_0}$ (see section \ref{sec_volume}),
the class $u_0$ is rigid, as we recalled in section \ref{sec_dynam}.

Let us consider the points $(I,u)$ and $(I_0,u_0)$ of the parabolic Teichm\"uller space $\TeichP^\circ$.
By Theorem \ref{thm_density} the $\MC^\circ$-orbit of $(I,u)$ is dense in $\TeichP^\circ$,
hence its closure contains $(I_0,u_0)$. Consider the universal deformation $\pi\colon\XX\to B$ of $X_{I_0}$,
where the base $B$ may be considered an open neighbourhood of $I_0$ in $\Teich^\circ$,
and $\XX_0 \simeq X_{I_0}$.
It follows from the density of the orbit of $(I,u)$ that there exists a sequence of points
$t_i\in B$, $t_i\to 0$, a sequence of elements $\mu_i\in \MC^\circ$ and a sequence of cohomology
classes $u_i\in H^2(X,\bbR)$ such that $\XX_{t_i} \simeq X_{\mu_i^* I}$, $u_i =\mu_i^* u$,
$u_i \to u_0$ when $i\to +\infty$. The classes $u_i$ are parabolic, in particular they are pseudo-effective.
Let $\delta(u)$ be the diameter of $u$ with respect to the canonical volume form, see section \ref{sec_volume}.
The action of $\MC^\circ$ preserves the diameter by Proposition \ref{prop_diam_invar}, so $\delta(u_i) = \delta(u)$. By Theorem \ref{thm_diam} we have
$0 = \delta(u_0) \ge \limsup_{i\to\infty} \delta(u_i) = \delta(u)$, hence $u$ is rigid.

{\em Case 2.} If $u^\perp\cap H^2(X,\bbQ)$ is spanned by a rational vector $v\in H^{2,0}(X)\oplus H^{0,2}(X)$,
we apply Theorem \ref{thm_orbit_closure}. We only need to check that the subdomain $\DomP^v\subset \DomP$
contains the period of a hyperk\"ahler manifold with hyperbolic automorphism. Since under
our assumptions $\dim(v^\perp)\ge 6$, we may apply \cite[Theorem 3.4]{AV2}. The construction
is the same as \cite[Corollary 3.5]{AV2}, but we apply it to the sublattice $H^2(X,\bbZ)\cap v^\perp$.
The rest of the proof is the same as in Case 1 above.

\subsection{Proof of Corollary \ref{cor_main}}\label{sec_cor_proof}
Denote $H = H^{1,1}_\bbR(X)$ and let $H_0$ be the $\bbR$-subspace of $H$ spanned by $H\cap H^2(X,\bbQ)$.
By the assumptions on $X$ we have $H_0\neq H$, otherwise $H^{2,0}(X)\oplus H^{0,2}(X)$ would be defined over $\bbQ$.
%Also, $H_0$ contains a positive vector, because $X$ is assumed projective.
%Let $H_1$ be the orthogonal complement of $H_0$ in $H$. It is clear that $q|_{H_1}$ is
%negative definite.

Let $v\in H$ be an arbitrary non-zero vector and $h\in H_0$ a K\"ahler class such that $h\notin \bbR v$.
Denote by $\CC^+\subset H$ the positive cone and by $\ell \subset H$ the affine line $\ell = h +\bbR v$.
Recall from section \ref{sec_Teich} that $\CC^+$ is cut into chambers by the hyperplanes
orthogonal to the MBM classes of type $(1,1)$, and the K\"ahler cone $\KK_X$ is one of the
chambers. Note that $\MBM^{1,1}\subset H_0$, hence for $x\in\MBM^{1,1}$ and $a\in\bbR$ we have $q(x,h+av) = q(x,h)\neq 0$,
because $h$ is a K\"ahler class. It follows that $\ell\cap x^\perp = \emptyset$ for any $x\in \MBM^{1,1}$,
hence $\ell\cap\CC^+\subset \KK_X$.

It is easy to check (and geometrically obvious) that there exists $a_0\in \bbR$ such that $u = h+a_0v\in \partial \CC^+$,
i.e.\ $q(u)=0$. It follows that $u\in \partial \KK_X$. Let $B_0\subset \KK_X$ be an open
neighbourhood of $h$. Then $B = B_0 + a_0 v$ is an open neighbourhood of $u$
and $U = B\cap \partial\CC^+$ is a non-empty open subset of $\partial\KK_X$.

By our assumptions $(H^{2,0}(X)\oplus H^{0,2}(X)) \cap H^2(X,\bbQ)$ is either zero- or one-dimensional.
Let $v$ be a generator of this $\bbQ$-subspace ($v=0$ if the subspace is trivial). 
It follows that for any $0\neq x\in (v^\perp\cap H^2(X,\bbQ))$ the intersection $H_x = x^\perp\cap H$
is a hyperplane in $H$. The hyperplane $H_x$ can not contain $U$, because $U$ is an
open subset of a non-degenerate quadratic cone. Hence $U_x = U\cap H_x$ is either empty or a
codimension one subset of $U$. Let $u\in U\setminus \cup\, U_x$, where
$x$ runs over all non-zero elements of $v^\perp\cap H^2(X,\bbQ)$. Then by Theorem \ref{thm_main}
the class $u$ is rigid. Such classes are clearly dense in $U$. This completes the proof.

\subsection{Proof of Theorem \ref{thm_main2}}\label{sec_main2_proof}

The same proof as for Theorem \ref{thm_main}, see section \ref{sec_main_proof} above,
using Theorem \ref{thm_density2} instead of Theorem \ref{thm_density}.

\subsection{Proof of Theorem \ref{_main3_Theorem_}}\label{sec_main3_proof}

We need the following lemma, which is proven in
\cite{_Filip_Tosatti:canonical_}.

\begin{lem}\label{_dichotomy_Lemma_}
Let  $V$ be a real quadratic vector space
of signature $(1,n)$ with $n \geq 2$, and $V_\Z$ a 
quadratic lattice in $V$. Consider a subgroup
$\Gamma$ of finite index in $\rmO(V_\Z)$,
and let $\6\CC$ be one of the two connected components of
the set of all non-zero isotropic vectors in $V$.
Then for any $\eta\in \6\CC$,
its $\Gamma$-orbit is dense when $\eta$ is irrational,
and closed otherwise.
\end{lem}

\begin{proof}
Let  $G=O(V)$, and $H\subset G$ be the stabilizer of
a vector $\eta\in \6\CC$. Then 
$\6\CC= G/H$. The action of $\Gamma$ on $G/H$ is
ergodic by Moore's theorem 
(\cite{_Morris:Ratner_}). The orbit dichotomy
follows from the Orbit Closure Dichotomy theorem, 
\cite[Theorem 7.4.2]{_Filip_Tosatti:canonical_},
where it is deduced from
the Raghunatan conjecture's solution
by Dani, \cite[Theorem A]{_Dani:Horospherical_}.
%See \cite[Theorem 7.4.2]{_Filip_Tosatti:canonical_} where the lemma is proved
%using the arguments similar to those used in section \ref{sec_teich}.
\end{proof}

We are in position to prove Theorem \ref{_main3_Theorem_}.
Since $H^{1,1}(X)$
contains no MBM classes, the group of complex
automorphisms $\Aut(X)$ acts on $H^2(X)$ as a
subgroup $\Gamma \subset \rmO(\NS)$ of finite index in $\rmO(\NS)$
(\cite{AV2}).
%where
%$\NS= H^{1,1}(X)\cap H^2(X, \Z)$ is the N\'eron-Severi
%lattice.
The group $\Gamma$ acts on the boundary $\6\Amp$
with two kinds of orbits by the orbit dichotomy theorem
(\ref{_dichotomy_Lemma_}): every rational orbit is closed,
and every irrational orbit is dense. Since $\Gamma$ is a
lattice in $\rmO(1,n)$, it contains a hyperbolic element;
the corresponding dynamical currents are rigid, and 
the corresponding cohomology classes have diameter zero 
(Theorem \ref{_dyn_current_rigid_Theorem_}). If an irrational
current $\eta$ is not rigid, it has positive diameter;
by semicontinuity of diameter, the same is true
for all classes $\eta_1\in \6\Amp$
obtained as limit points of $\Gamma \cdot \eta$.
However, the set $\Gamma \cdot \eta$ is dense,
hence it contains dynamical currents, giving
a contradiction. This finishes the proof of Theorem 
\ref{_main3_Theorem_}.

\noindent {\sc Andrey Soldatenkov\\
{\sc Instituto Nacional de Matem\'atica Pura e
              Aplicada (IMPA) \\ Estrada Dona Castorina, 110\\
Jardim Bot\^anico, CEP 22460-320\\
Rio de Janeiro, RJ - Brasil}\\
{\tt aosoldatenkov@gmail.com}}

\hfill

\noindent {\sc Misha Verbitsky\\
{\sc Instituto Nacional de Matem\'atica Pura e
              Aplicada (IMPA) \\ Estrada Dona Castorina, 110\\
Jardim Bot\^anico, CEP 22460-320\\
Rio de Janeiro, RJ - Brasil\\
\tt  verbit@impa.br }}

\end{document}